\documentclass{amsart}%
\usepackage{amsmath}
\usepackage{amssymb}
\usepackage{graphicx}
\usepackage{amscd}%
\usepackage{amsfonts}
\newtheorem{theorem}{Theorem}
\theoremstyle{plain}

\newtheorem{corollary}{Corollary}

\newtheorem{definition}{Definition}
\newtheorem{example}{Example}

\newtheorem{lemma}{Lemma}
\newtheorem{notation}{Notation}

\newtheorem{proposition}{Proposition}
\newtheorem{remark}{Remark}

\numberwithin{equation}{section}

\begin{document}
\title[Deformations of Levi flats]{Deformations of Levi flat hypersurfaces in complex manifolds}
\author{Paolo de Bartolomeis}
\address{Universit\`{a} di Firenze\\
Dipartimento di Matematica Applicata "G. Sansone" \\
Via di Santa Marta 3 \\
I-50139 Firenze, Italia }
\email{paolo.debartolomeis@unifi.it}
\author{Andrei Iordan}
\address{Institut de Math\'{e}matiques \\
UMMR 7586 du CNRS, case 247 \\
Universit\'{e} Pierre et Marie-Curie \\
4 Place Jussieu \\
75252 Paris Cedex 05\\
France}
\email{iordan@math.jussieu.fr}
\date{March, 31, 2011}

\begin{abstract}
We first give a deformation theory of integrable distributions of codimension
$1$. This theory is used to study Levi-flat deformations: a Levi-flat
deformation of a Levi flat hypersurface $L$ in a complex manifold is a smooth
mapping $\Psi:I\times M\rightarrow M$ such that $\Psi_{t}=\Psi\left(
t,\cdot\right)  \in Diff\left(  M\right)  $,\ $L_{t}=\Psi_{t}L$ is a Levi flat
hypersurface in $M$ for every $t\in I$ \ and $L_{0}=L$. We define a
parametrization of families of smooth hypersurfaces near $L$ such that the
Levi flat deformations are given by the solutions of the Maurer-Cartan
equation in a DGLA associated to the Levi foliation. We say that $L$ is
infinitesimally rigid if the tangent cone at the origin to the moduli space of
Levi flat deformations of $L$ is trivial. We prove the infinitesimal rigidity
of compact transversally parallelisable Levi flat hypersurfaces in compact
complex manifolds and give sufficient conditions for infinitesimal rigidity in
K\"{a}hler manifolds. As an application, we prove the nonexistence of
transversally parallelizable Levi flat hypersurfaces in a class of manifolds
which contains $\mathbb{CP}_{2}$.

\end{abstract}
\maketitle

\section{\bigskip Introduction}

Let $M$ be a complex manifold and $L$ a real hypersurface of class $C^{2}$ in
$M$ such that $M\backslash L=$ $\Omega_{1}\cup\Omega_{2}$, $\Omega_{1}%
\cap\Omega_{2}=\emptyset$. $L$ is Levi flat if it satisfies one of the
following equivalent conditions:

1) $\Omega_{1}$ and $\Omega_{2}$ are pseudoconvex domains.

2) $L$ is foliated by complex hypersurfaces of $M$.

3) The Levi form of $L$ vanishes.

It is well known that in general, if $L$ is not of class $C^{2}$, we have only
$3)\implies2)\implies1)$.

One of the oldest result concerning Levi flat hypersurfaces is a theorem of
E.\ Cartan \cite{Cartan32} which states that a real analytic Levi flat
hypersurface is locally isomorphic to the set of vanishing of the real part of
a holomorphic function. A generalization of this theorem for singular Levi
flat hypersurfaces can be found in \cite{Gong.99}.

Recent research on Levi flat hypersurfaces in complex manifolds were motivated
by the following conjecture of D.\ Cerveau \cite{Cerveau93}: there are no
smooth Levi flat hypersurfaces in the complex projective space $\mathbb{CP}%
_{n}$, $n\geqslant2$.

For $n\geqslant3$, this conjecture was proved by Lins Neto for real analytic
Levi flat hypersurfaces \cite{Neto99}, by Y.-T. Siu for Levi flat
hypersurfaces of class $C^{12}$ \cite{Siu00} and by A.\ Iordan and F. Matthey
for Lipschitz hypersurfaces of Sobolev class $W^{s}$, $s>5/2$ \cite{Iordan08}.
Despite several attempts to prove this conjecture for $n=2$, its proof is
still incomplete.

Unlike $\mathbb{CP}_{n}$, $n\geqslant2$, the complex tori $\mathbb{T}%
_{n}=\mathbb{C}^{n}/\Gamma$ contains the Levi flat hypersurfaces $\pi\left(
\oplus_{j=1}^{2n-1}\mathbb{R}u_{j}+u\right)  $ where $\pi:\mathbb{C}%
^{n}\rightarrow\mathbb{T}_{n}$ is the canonical projection, $u_{j},$
$j=1,\cdot\cdot\cdot,2n-1$ are $\mathbb{R}$-linearly independent vectors in
$\Gamma$ and $u\in\mathbb{C}^{n}$ \cite{Matsumoto02}. It was conjectured in
\cite{Matsumoto02} that for every compact Levi flat hypersurface $M$ in
$\mathbb{T}_{n}$, $\pi^{-1}\left(  M\right)  $ is a union of affine hyperplanes.

In this paper we study the deformations of smooth Levi flat hypersurfaces in
complex manifolds. The theory of deformations of complex manifolds was
intensively studied from the 50s beginning with the famous results of Kodaira
and Spencer \cite{KodairaSpencer58} (see for ex. \cite{Kodaira05}%
,\ \cite{Voisin02}). In \cite{Nijenhuis66}, Nijenhuis ans Richardson adapted a
theory initiated by \ Gerstenhaber \cite{Gerstenhaber64} \ and proved the
connection between the deformations of complex analytic structures and the
theory of differential graded Lie algebras (DGLA). This theory was developped
following ideas of Deligne by Goldman and Millson \cite{Goldman88}.

The main results of this paper may be summarized as follows.

In the first chapter we consider integrable distributions of codimension 1 on
smooth manifolds and we define a DGLA\ associated to the foliation such that
the deformations of integrable distributions of codimension 1 are given by
solutions of Maurer-Cartan equation in this algebra. As the examples show,
this theory is highly non trivial and it seems to be interesting by itself. We
mention that Kodaira and Spencer developped in \cite{Kodaira61} a theory of
deformations of the so called multifoliate structures, which are more general
then the foliate structures. Our approach in this paper for foliations of
codimension $1$ is different of theirs (see Remark
\ref{Comparaison avec Kodaira Spencer}) and allows us to study the Levi flat case.

In the second chapter we give a description of the deformations of a smooth
Levi flat hypersurface $L$ in a complex manifold by means of the Maurer-Cartan
equation in the DGLA associated to the Levi foliation.

Then we establish the equations verified by the tangent to a regular familly
of Levi flat deformations. We say that $L$ is infinitesimally rigid
(respectively strongly infinitesimally rigid) if the tangent cone at the
origin to the moduli space of Levi flat deformations of $L$ is trivial
(respectively if the tangent cone at the origin to the solutions of the
Maurer-Cartan equation in the DGLA\ associated to the Levi foliation is
trivial) . We remark that Diederich and Ohsawa study in \cite{Diederich07} the
displacement rigidity of Levi flat hypersurfaces in disc bundle over compact
Riemann surfaces. The definition of rigidity in \cite{Diederich07} means that
any small $C^{2}$ perturbation of a Levi flat hypersurface $L$ is CR
isomorphic with $L,$ so $L$ is strongly infinitesimally rigid.

We prove that a transversally parallelizable compact Levi flat hypersurface in
a compact complex manifold is strongly infinitesimally rigid and we give a
sufficient condition for infinitesimal rigidity in K\"{a}hler manifolds
(Theorem \ref{rigidity}). As an application, we prove that there are no
compact transversally parallelizable Levi flat hypersurfaces in connected
complex manifolds $M$ such that for every $p\neq q\in M$ and every real
hyperplane $H_{q}$ in $T_{q}M$ there exists a holomorphic vector field $Y$ on
$M$ such that $Y\left(  p\right)  =0$ and $Y\left(  q\right)  \oplus
H_{q}=T_{q}M$. If $M=\mathbb{CP}_{n}$, $n\geq2$, the hypothesis of the
previous result are fulfilled.

The non existence of transversally parallelizable Levi flat hypersurfaces in
$\mathbb{CP}_{2}$ can be obtained by different proofs. We chose here to give a
proof by using the results of this paper. Another direct proof was furnished
to the authors by Marco Brunella \cite{Brunella10} who disappierd recently in
a tragic accident. We want to pay tribute to the memory of Marco Brunella by
giving also his proof of this result.

\section{Deformation theory of integrable distribution of codimension 1}

\subsection{DGLA associated to an integrable distribution of codimension
1\bigskip\protect\linebreak}

\ \ \ \ \ \ \ \ \ \ \ \ \ \ \ \newline

\begin{definition}
A differential graded Lie agebra (DGLA) is a triple $\left(  V^{\ast
},d,\left[  \cdot,\cdot\right]  \right)  $ such that:

1) $V^{\ast}=\oplus_{i\in\mathbb{N}}V^{i}$, where $\left(  V^{i}\right)
_{i\in\mathbb{N}}$ \ is a family of $\mathbb{C}$-vector spaces and $d:V^{\ast
}\rightarrow V^{\ast}$is a graded homomorphism such that $d^{2}=0$. An element
$a\in V^{k}$ is said to be homogeneous of degree $k=\deg a$.

2) $\left[  \cdot,\cdot\right]  :$ $V^{\ast}\times V^{\ast}\rightarrow
V^{\ast}$defines a structure of graded Lie algebra i.e. for homogeneous
elements we have%

\begin{equation}
\left[  a,b\right]  =-\left(  -1\right)  ^{\deg a\deg b}\left[  b,a\right]
\label{antisym}%
\end{equation}
and
\begin{equation}
\left[  a,\left[  b,c\right]  \right]  =\left[  \left[  a,b\right]  ,c\right]
+\left(  -1\right)  ^{\deg a\deg b}\left[  b,\left[  a,c\right]  \right]
\label{Jacobi}%
\end{equation}

3) $d$ is compatible with the graded Lie algebra structure i.e.
\begin{equation}
d\left[  a,b\right]  =\left[  da,b\right]  +\left(  -1\right)  ^{\deg
a}\left[  a,db\right]  .\label{d(.)}%
\end{equation}

\end{definition}

\begin{remark}
If (\ref{antisym}) is satisfied then (\ref{Jacobi}) is equivalent to%
\begin{equation}
\mathfrak{S}_{s}\left(  -1\right)  ^{\deg a\deg c}\left[  a,\left[
b,c\right]  \right]  =0\label{Jacobi'}%
\end{equation}
where $\mathfrak{S}_{s}$ denotes the symmetric sum.
\end{remark}

\begin{definition}
Let $\left(  V^{\ast},d,\left[  \cdot,\cdot\right]  \right)  $ be a DGLA and
$a\in V^{1}$. We say that $a$ verifies the Maurer Cartan equation in $\left(
V^{\ast},d,\left[  \cdot,\cdot\right]  \right)  $ if%
\begin{equation}
da+\frac{1}{2}\left[  a,a\right]  =0.\label{MC}%
\end{equation}

\end{definition}

\begin{lemma}
\label{d2a}Let $\left(  V^{\ast},d,\left[  \cdot,\cdot\right]  \right)  $ be a
DGLA and $a\in V^{1}$ verifying the Maurer Cartan equation (\ref{MC}). Set
$d_{a}=d+\left[  a,\cdot\right]  $. Then for every $\omega\in V^{\ast}$ we
have%
\[
d_{a}^{2}\omega=\left[  da+\frac{1}{2}\left[  a,a\right]  ,\omega\right]  .
\]

\end{lemma}

\begin{proof}
Let $\omega\in V^{k}$. Since $d$ satisfies (\ref{d(.)}) we have
\begin{align*}
d_{a}^{2}\omega & =\left(  d+\left[  a,\cdot\right]  \right)  \left(
d\omega+\left[  a,\omega\right]  \right)  =d\left[  a,\omega\right]  +\left[
a,d\omega\right]  +\left[  a,\left[  a,\omega\right]  \right] \\
& =\left[  da,\omega\right]  -\left[  a,d\omega\right]  +\left[
a,d\omega\right]  +\left[  a,\left[  a,\omega\right]  \right] \\
& =\left[  da,\omega\right]  +\left[  a,\left[  a,\omega\right]  \right]  .
\end{align*}
But (\ref{Jacobi}) give%
\[
\left[  a,\left[  a,\omega\right]  \right]  =\frac{1}{2}\left[  \left[
a,a\right]  ,\omega\right]
\]
and the lemma follows.
\end{proof}

From Lemma \ref{d2a} we obtain the following

\begin{corollary}
\label{d2a=0}Let $\left(  V^{\ast},d,\left[  \cdot,\cdot\right]  \right)  $ be
a DGLA and $a\in V^{1}$ verifying the Maurer Cartan equation (\ref{MC}). Then
$d_{a}^{2}=0$. Moreover, if $Z\left(  V^{\ast}\right)  =\left\{  0\right\}  $,
where $Z\left(  V^{\ast}\right)  =\left\{  \beta\in V^{\ast}:\ \left[
\beta,\alpha\right]  =0,\ \forall\alpha\in V^{\ast}\right\}  $ is the center
of $\left(  V^{\ast},d,\left[  \cdot,\cdot\right]  \right)  $, then $a$
verifies Maurer Cartan equation (\ref{MC}) if and only if $d_{a}^{2}=0$ .
\end{corollary}

The starting point of the theory developped in this section is the following:

\begin{lemma}
\label{Forms=DGLA}Let $L$ be a $C^{\infty}$ manifold and $X$ a vector field on
$L$. We denote by $\Lambda^{k}\left(  L\right)  $ the $k$-forms on $L$ and
$\Lambda^{\ast}\left(  L\right)  =\oplus_{k\in\mathbb{N}}\Lambda^{k}\left(
L\right)  $. For $\alpha,\beta\in\Lambda^{\ast}\left(  L\right)  $, set%
\begin{equation}
\left\{  \alpha,\beta\right\}  =\mathcal{L}_{X}\alpha\wedge\beta-\alpha
\wedge\mathcal{L}_{X}\beta\label{Lie braket 1}%
\end{equation}
where $\mathcal{L}_{X}$ is the Lie derivative. Then $\left(  \Lambda^{\ast
}\left(  L\right)  ,d,\left\{  \cdot,\cdot\right\}  \right)  $ is a DGLA.
\end{lemma}

\begin{proof}
\ Since (\ref{antisym}) is obvious we will verify (\ref{Jacobi'}). We have%
\begin{align*}
\mathfrak{S}_{s}\left(  -1\right)  ^{\deg a\deg c}\left\{  a,\left\{
b,c\right\}  \right\}   & =\mathfrak{S}_{s}\left(  -1\right)  ^{\deg a\deg
c}(\mathcal{L}_{X}a\wedge\mathcal{L}_{X}b\wedge c\\
& -\mathcal{L}_{X}a\wedge b\wedge\mathcal{L}_{X}c-a\wedge\mathcal{L}_{X}%
^{2}b\wedge c+a\wedge b\wedge\mathcal{L}_{X}^{2}c).
\end{align*}
Since
\[
\left(  -1\right)  ^{\deg c\deg a}\mathcal{L}_{X}a\wedge\mathcal{L}_{X}b\wedge
c=\left(  -1\right)  ^{\deg a\deg b}\mathcal{L}_{X}b\wedge c\wedge
\mathcal{L}_{X}a
\]
and%
\[
\left(  -1\right)  ^{\deg a\deg c}a\wedge\mathcal{L}_{X}^{2}b\wedge c=\left(
-1\right)  ^{\deg b\deg c}c\wedge a\wedge\mathcal{L}_{X}^{2}b
\]
it follows that%
\[
\mathfrak{S}_{s}\left(  -1\right)  ^{\deg a\deg c}\left\{  a,\left\{
b,c\right\}  \right\}  =0.
\]
By using Cartan's formula
\[
\mathcal{L}_{X}=\iota_{X}d+d\iota_{X}%
\]
we obtain%
\begin{align*}
d\left\{  a,b\right\}   & =d\left(  \left(  \iota_{X}d+d\iota_{X}\right)
a\wedge b-a\wedge\left(  \iota_{X}d+d\iota_{X}\right)  b\right) \\
& =d\iota_{X}da\wedge b+\left(  -1\right)  ^{\deg a}\iota_{X}da\wedge
db+\left(  -1\right)  ^{\deg a}d\iota_{X}a\wedge db\\
& -da\wedge\iota_{X}db-da\wedge d\iota_{X}b-\left(  -1\right)  ^{\deg
a}a\wedge d\iota_{X}db\\
& =\left\{  da,b\right\}  +\left(  -1\right)  ^{\deg a}\left\{  a,db\right\}
.
\end{align*}

\end{proof}

\begin{lemma}
\label{Frobenius} Let $L$ be a $C^{\infty}$ manifold and $\xi\subset T\left(
L\right)  $ a distribution of codimension $1$. Let $\gamma\in\wedge^{1}\left(
L\right)  $ such that $\ker\gamma=\xi$ and $X$ a vector field on $L$ such that
$\gamma\left(  X\right)  =1$. Then the following are equivalent:

i) $\xi$ is integrable;

ii) There exists $\alpha\in\wedge^{1}\left(  L\right)  $ such that
$d\gamma=\alpha\wedge\gamma$;

iii) $d\gamma\wedge\gamma=0$;

iv) $d\gamma=-\iota_{X}d\gamma\wedge\gamma$;

v) $\gamma$ satisfies the Maurer Cartan equation (\ref{MC}) in $\left(
\Lambda^{\ast}\left(  L\right)  ,d,\left\{  \cdot,\cdot\right\}  \right)  $,
where $\left\{  \cdot,\cdot\right\}  $ is defined in (\ref{Lie braket 1}).
\end{lemma}

\begin{proof}
$\bigskip ii)\Rightarrow iii)$ and $iv)\Rightarrow ii)$ are evident.

$iii)\Rightarrow iv)$ Suppose
\[
d\gamma\wedge\gamma=0.
\]

Since
\[
\iota_{X}\left(  a\wedge b\right)  =\iota_{X}a\wedge b+\left(  -1\right)
^{\deg\left(  a\right)  }a\wedge\iota_{X}b,\ a,b\in\Lambda^{\ast}\left(
L\right)  ,
\]
we have
\[
0=\iota_{X}\left(  d\gamma\wedge\gamma\right)  =\iota_{X}\left(
d\gamma\right)  \wedge\gamma+\left(  \iota_{X}\gamma\right)  d\gamma=\iota
_{X}\left(  d\gamma\right)  \wedge\gamma+d\gamma,
\]
and so%
\[
d\gamma=-\iota_{X}\left(  d\gamma\right)  \wedge\gamma.
\]

$iv)\Leftrightarrow v)$ Since $\iota_{X}\gamma=1$ we have
\[
\left\{  \gamma,\gamma\right\}  =\mathcal{L}_{X}\gamma\wedge\gamma
-\gamma\wedge\mathcal{L}_{X}\gamma=\iota_{X}d\gamma\wedge\gamma-\gamma
\wedge\iota_{X}d\gamma=2\iota_{X}d\gamma\wedge\gamma
\]
so%
\[
d\gamma+\frac{1}{2}\left\{  \gamma,\gamma\right\}  =d\gamma+\iota_{X}%
d\gamma\wedge\gamma.
\]

As $i)\Leftrightarrow ii)$ is the theorem of Frobenius, the Lemma is proved.
\end{proof}

By Lemma \ref{Forms=DGLA}, Lemma \ref{Frobenius} and Corollary \ref{d2a=0} we obtain

\begin{corollary}
\label{Forms+delta=DGLA}Let $L$ be a $C^{\infty}$ manifold and $\xi\subset
T\left(  L\right)  $ an integrable distribution of codimension $1$. Let
$\gamma\in\wedge^{1}\left(  L\right)  $ such that $\ker\gamma=\xi$ and $X$ a
vector field on $L$ such that $\gamma\left(  X\right)  =1$. Set
\[
\delta=d_{\gamma}=d+\left\{  \gamma,\cdot\right\}
\]
where $\left\{  \cdot,\cdot\right\}  $ is defined in (\ref{Lie braket 1}).
Then $\left(  \Lambda^{\ast}\left(  L\right)  ,\delta,\left\{  \cdot
,\cdot\right\}  \right)  $ is a DGLA.
\end{corollary}

\begin{remark}
Let $Z\left(  \Lambda^{\ast}\left(  L\right)  \right)  $ be the center of
$\left(  \Lambda^{\ast}\left(  L\right)  ,\delta,\left\{  \cdot,\cdot\right\}
\right)  $. Then $Z\left(  \Lambda^{\ast}\left(  L\right)  \right)  =\left\{
0\right\}  $. Indeed, let $\alpha\in Z\left(  \Lambda^{\ast}\left(  L\right)
\right)  $. \ Since $\left\{  \alpha,1\right\}  =\mathcal{L}_{X}\alpha$ it
follows that $\mathcal{L}_{X}\alpha=0$. Let $x\in L$ and choose local
coordinates $\left(  x_{1},\cdot\cdot\cdot,x_{n}\right)  $ in a neighborhood
$U$ of $x$ such that $X=\frac{\partial}{\partial x_{1}}$ on $U$. Let $\beta
\in\Lambda^{0}\left(  L\right)  $ such that $\beta=x_{1}$ in a neighborhood of
$x$. Then%
\[
\left\{  \alpha,\beta\right\}  \left(  X\right)  =\left(  \mathcal{L}%
_{X}\alpha\wedge\beta-\alpha\wedge\mathcal{L}_{X}\beta\right)  \left(
X\right)  =-\alpha\left(  X\right)  =0
\]
and so $\alpha=0$.
\end{remark}

\begin{corollary}
\label{Z* subalgebra}Under the hypothesis of Corollary \ref{Forms+delta=DGLA},
we set
\[
\mathcal{Z}^{\ast}\left(  L\right)  =\left\{  \alpha\in\Lambda^{\ast}\left(
L\right)  :\ \iota_{X}\alpha=0\right\}  .
\]
Then $\left(  \mathcal{Z}^{\ast}\left(  L\right)  ,\delta,\left\{  \cdot
,\cdot\right\}  \right)  $ is a sub-DGLA of $\left(  \Lambda^{\ast}\left(
L\right)  ,\delta,\left\{  \cdot,\cdot\right\}  \right)  $.
\end{corollary}

\begin{proof}
Let $\alpha,\beta\in\mathcal{Z}^{\ast}\left(  L\right)  $. Since$\ \ \iota
_{X}\alpha=0$, $\iota_{X}\beta=0$ and $\ \iota_{X}^{2}=0$ we have%
\[
\iota_{X}\delta\alpha=\iota_{X}\left(  d\alpha+\iota_{X}d\gamma\wedge
\alpha-\gamma\wedge\iota_{X}d\alpha\right)  =\iota_{X}d\alpha-\iota_{X}%
d\alpha=0
\]
and
\begin{align*}
\ \iota_{X}\left\{  \alpha,\beta\right\}   & =\ \iota_{X}\left(
\mathcal{L}_{X}\alpha\wedge\beta-\alpha\wedge\mathcal{L}_{X}\beta\right)
=\iota_{X}\mathcal{L}_{X}\alpha\wedge\beta-\left(  -1\right)  ^{\deg\alpha
}\alpha\wedge\iota_{X}\mathcal{L}_{X}\beta\\
& =\iota_{X}\left(  \iota_{X}d+d\iota_{X}\right)  \alpha\wedge\beta-\left(
-1\right)  ^{\deg\alpha}\alpha\wedge\iota_{X}\left(  \iota_{X}d+d\iota
_{X}\right)  \beta=0.
\end{align*}

\end{proof}

\begin{remark}
Let $L$ be a $C^{\infty}$ manifold and $\xi\subset T\left(  L\right)  $ an
integrable distribution of codimension $1$. Then there exists a $1$-form
$\gamma$ on $L$ such that $\xi=\ker~\gamma$ if and only if $\xi$ is
co-orientable, i.e. the normal space to the foliation defined by $\xi$ is
orientable (see for ex. \cite{Godbillon91}).
\end{remark}

\begin{definition}
Let $L$ be a $C^{\infty}$ manifold and $\xi\subset T\left(  L\right)  $ an
integrable co-orientable distribution of codimension $1$. A couple $\left(
\gamma,X\right)  $ where $\gamma\in\wedge^{1}\left(  L\right)  $ and $X$ is a
vector field on $L$ such that $\ker~\gamma=\xi$ and $\gamma\left(  X\right)
=1$ will be called a DGLA defining couple.
\end{definition}

\begin{remark}
Let $L$ be a $C^{\infty}$ manifold and $\xi\subset T\left(  L\right)  $ an
integrable distribution of codimension $1$. Let $\left(  \gamma,X\right)  $ be
a DGLA defining couple for an integrable distribution $\xi$ of codimension
$1$. Then $\left(  \gamma^{\prime},X^{\prime}\right)  $ is a DGLA defining
couple for $\xi$ if and only if $\gamma^{\prime}=e^{\lambda}\gamma$,
$\lambda\in C^{\infty}\left(  M\right)  $ and $X^{\prime}=e^{-\lambda}X+V$,
$V\in\xi$. Compare with the contact distribution case: the existence of a
contact form $\omega$ on a odd dimensional manifold is equivalent with the
co-orientability of \ the contact distribution \cite{Grey59} and it is unique
up to a multiplication with a nonvanishing function. In this case the Reeb
vector field $R$ is uniquely defined by $\iota_{R}\omega=1$ and $\iota
_{R}d\omega=0$. But contact distributions are nonintegrable.
\end{remark}

\begin{remark}
Let $\alpha,\beta\in\mathcal{Z}^{\ast}\left(  L\right)  $ and $\left(
\gamma,X\right)  $ a DGLA defining couple. Then%
\begin{equation}
\left\{  \alpha,\beta\right\}  =\left(  \iota_{X}d+d\iota_{X}\right)
\alpha\wedge\beta-\alpha\wedge\left(  \iota_{X}d+d\iota_{X}\right)
\beta=\iota_{X}d\alpha\wedge\beta-\alpha\wedge\iota_{X}d\beta\label{alfa beta}%
\end{equation}
and%
\begin{equation}
\left\{  \gamma,\alpha\right\}  =\left(  \iota_{X}d+d\iota_{X}\right)
\gamma\wedge\alpha-\gamma\wedge\left(  \iota_{X}d+d\iota_{X}\right)
\alpha=\iota_{X}d\gamma\wedge\alpha-\gamma\wedge\iota_{X}d\alpha
.\label{gama alfa}%
\end{equation}

\end{remark}

\begin{definition}
Let $\left(  V^{\ast},d_{V},\left[  \cdot,\cdot\right]  _{V}\right)  $,
$\left(  W^{\ast},d_{W},\left[  \cdot,\cdot\right]  _{W}\right)  $ be DGLA and
$\Phi:V^{\ast}\rightarrow W^{\ast}$ a graded morphism. We say that $\Phi$ is a
DGVS-morphism (differential graded vector space morphism) if $\Phi d_{V}%
=d_{W}\Phi$. A DGVS-morphism$\ \Phi$ is a DGLA-morphism if $\left[
\Phi\left(  \alpha\right)  ,\Phi\left(  \beta\right)  \right]  _{W}%
=\Phi\left(  \left[  \alpha,\beta\right]  _{V}\right)  $ for every
$\alpha,\beta\in V^{\ast}$.
\end{definition}

\begin{remark}
The DGLA structure of $\mathcal{Z}^{\ast}\left(  L\right)  $ depends on the
choice of the DGLA defining couple $\left(  \gamma,X\right)  $. In what
follows, for given $\xi$ we will fix $\gamma$ and $X$. When it is necessary to
emphasize this dependence we will write $\left(  \mathcal{Z}_{\gamma,X}^{\ast
}\left(  L\right)  ,\delta_{\gamma,X},\left\{  \cdot,\cdot\right\}
_{\gamma,X}\right)  $.
\end{remark}

The following Proposition will describe shortly the effects of changing the
defining couple:

\begin{proposition}
\label{Iso DGVS}Let $L$ be a $C^{\infty}$ manifold and $\xi\subset T\left(
L\right)  $ an integrable distribution of codimension $1$. Let $\left(
\gamma,X\right)  $ be a DGLA defining couple, $V$ a $\xi$-valued vector field
and $\lambda\in C^{\infty}\left(  L\right)  $. For $\alpha\in\mathcal{Z}%
^{\ast}\left(  L\right)  $ consider $\Psi\left(  \alpha\right)  =\Psi
_{\lambda}\left(  \alpha\right)  =e^{\lambda}\alpha$ and $\Theta\left(
\alpha\right)  =\Theta_{V}\left(  \alpha\right)  =\alpha+\left(  -1\right)
^{\deg\alpha}\iota_{V}\alpha\wedge\gamma$. Then:

i) $\Psi:\left(  \mathcal{Z}_{\gamma,X}^{\ast}\left(  L\right)  ,\delta
_{\gamma,X},\left\{  \cdot,\cdot\right\}  _{\gamma,X}\right)  \rightarrow
\left(  \mathcal{Z}_{e^{\lambda}\gamma,e^{-\lambda}X}^{\ast}\left(  L\right)
,\delta_{e^{\lambda}\gamma,e^{-\lambda}X},\left\{  \cdot,\cdot\right\}
_{e^{\lambda}\gamma,e^{-\lambda}X}\right)  $ is a DGLA-isomorphism.

ii) $\Theta:\left(  \mathcal{Z}_{\gamma,X}^{\ast}\left(  L\right)
,\delta_{\gamma,X}\right)  \rightarrow\left(  \mathcal{Z}_{\gamma,X+V}^{\ast
}\left(  L\right)  ,\delta_{\gamma,X+V}\right)  $ is a DGVS-isomorphism.
\end{proposition}

\begin{proof}
i) Let $\alpha,\beta\in\mathcal{Z}_{\gamma,X}^{\ast}\left(  L\right)  $. By
(\ref{alfa beta}) and (\ref{gama alfa}) we have%
\begin{equation}
\Psi\delta_{\gamma,X}\alpha=e^{\lambda}\left(  da+\left\{  \gamma
,\alpha\right\}  _{\gamma,X}\right)  =e^{\lambda}\left(  da+\iota_{X}%
d\gamma\wedge\alpha-\gamma\wedge\iota_{X}d\alpha\right) \label{csidelta}%
\end{equation}
and%
\begin{align}
\left\{  e^{\lambda}\gamma,e^{\lambda}\alpha\right\}  _{e^{\lambda}%
\gamma,e^{-\lambda}X}  & =\iota_{e^{-\lambda}X}d\left(  e^{\lambda}%
\gamma\right)  \wedge e^{\lambda}\alpha-e^{\lambda}\gamma\wedge\iota
_{e^{-\lambda}X}d\left(  e^{\lambda}\alpha\right) \nonumber\\
& =\iota_{X}\left(  e^{\lambda}d\lambda\wedge\gamma+e^{\lambda}d\gamma\right)
\wedge\alpha-\gamma\wedge\iota_{X}d\left(  e^{\lambda}\alpha\right)
\nonumber\\
& =e^{\lambda}[\iota_{X}\left(  d\lambda\right)  \gamma\wedge\alpha
-d\lambda\wedge\alpha+\iota_{X}\left(  d\gamma\right)  \wedge\alpha
-\gamma\wedge\iota_{X}\left(  d\lambda\wedge\alpha\right) \nonumber\\
& -\iota_{X}\left(  d\lambda\right)  \gamma\wedge\alpha-\gamma\wedge\iota
_{X}d\alpha]\nonumber\\
& =e^{\lambda}\left[  -d\lambda\wedge\alpha+\iota_{X}\left(  d\gamma\right)
\wedge\alpha-\gamma\wedge\iota_{X}d\alpha\right]  .\label{gama e lambda}%
\end{align}
By replacing (\ref{gama e lambda}) in the formula
\[
\delta_{e^{\lambda}\gamma,e^{-\lambda}X}\Psi\alpha=d\left(  e^{\lambda}%
\alpha\right)  +\left\{  \gamma,e^{\lambda}\alpha\right\}  _{e^{\lambda}%
\gamma,e^{-\lambda}X},
\]
we deduce from (\ref{csidelta}) that
\[
\Psi\delta_{\gamma,X}=\delta_{e^{\lambda}\gamma,e^{-\lambda}X}\Psi.
\]

We have also%
\begin{align*}
\left\{  \Psi\alpha,\Psi\left(  \beta\right)  \right\}  _{e^{\lambda}%
\gamma,e^{-\lambda}X}  & =\left\{  e^{\lambda}\alpha,e^{\lambda}\beta\right\}
_{e^{\lambda}\gamma,e^{-\lambda}X}=\iota_{e^{-\lambda}X}d\left(  e^{\lambda
}\alpha\right)  \wedge e^{\lambda}\beta-e^{\lambda}\alpha\wedge\iota
_{e^{-\lambda}X}d\left(  e^{\lambda}\beta\right) \\
& =e^{\lambda}\left[  \iota_{X}\left(  d\lambda\wedge\alpha+d\alpha\right)
\wedge\beta-\alpha\wedge\iota_{X}\left(  d\lambda\wedge\beta+d\beta\right)
\right] \\
& =e^{\lambda}\left[  \iota_{X}\left(  d\lambda\right)  \alpha\wedge
\beta+\iota_{X}d\alpha\wedge\beta-\iota_{X}\left(  d\lambda\right)
\alpha\wedge\beta-\alpha\wedge\iota_{X}d\beta\right] \\
& =e^{\lambda}\left[  \iota_{X}d\alpha\wedge\beta-\alpha\wedge\iota_{X}%
d\beta\right]  =\Psi\left\{  \alpha,\beta\right\}  _{\gamma,X}%
\end{align*}

ii) Let $\alpha\in$ $\mathcal{Z}_{\gamma,X}^{\ast}\left(  L\right)  $. Then
\begin{align*}
\iota_{X+V}\Theta\alpha & =\iota_{X+V}\left(  \alpha+\left(  -1\right)
^{\deg\alpha}\iota_{V}\alpha\wedge\gamma\right) \\
& =\iota_{V}\alpha+\left(  -1\right)  ^{\deg\alpha}\iota_{X}\left(  \iota
_{V}\alpha\wedge\gamma\right)  +\left(  -1\right)  ^{\deg\alpha}\iota
_{V}\left(  \iota_{V}\alpha\wedge\gamma\right) \\
& =\iota_{V}\alpha+\left(  -1\right)  ^{\deg\alpha}\iota_{X}\iota_{V}%
\alpha\wedge\gamma-\iota_{V}\alpha=0.
\end{align*}
It follows that $\Theta$ is well defined and the map $\Theta^{\prime
}:\mathcal{Z}_{\gamma,X+V}^{\ast}\left(  L\right)  \rightarrow\mathcal{Z}%
_{\gamma,X}^{\ast}\left(  L\right)  $ defined by $\Theta^{\prime}\left(
\alpha\right)  =\alpha+\left(  -1\right)  ^{\deg\alpha}\iota_{-V}\alpha
\wedge\gamma$ is the inverse of $\Theta$.

Since $\iota_{V}\gamma=0$ and $d\gamma=$ $-\iota_{X}d\gamma\wedge\gamma,$ by
using the expression of $\delta_{\gamma,X}$ from (\ref{csidelta}), we obtain%
\begin{align}
\Theta\delta_{\gamma,X}\alpha & =\delta_{\gamma,X}\alpha-\left(  -1\right)
^{\deg\alpha}\iota_{V}\left(  d\alpha+\iota_{X}d\gamma\wedge\alpha
-\gamma\wedge\iota_{X}d\alpha\right)  \wedge\gamma\nonumber\\
& =\delta_{\gamma,X}\alpha-\left(  -1\right)  ^{\deg\alpha}\iota_{V}%
d\alpha\wedge\gamma-\left(  -1\right)  ^{\deg\alpha}\left(  \iota_{V}\iota
_{X}d\gamma\right)  \wedge\alpha\wedge\gamma\nonumber\\
& +\left(  -1\right)  ^{\deg\alpha}\iota_{X}d\gamma\wedge\iota_{V}\alpha
\wedge\gamma\nonumber\\
& =\delta_{\gamma,X}\alpha-\gamma\wedge\iota_{V}d\alpha-\left(  -1\right)
^{\deg\alpha}\left(  \iota_{V}\iota_{X}d\gamma\right)  \wedge\alpha
\wedge\gamma\label{fidelta}\\
& +d\gamma\wedge\iota_{V}\alpha.\nonumber
\end{align}
We have%
\begin{align*}
\left\{  \gamma,\alpha+\left(  -1\right)  ^{\deg\alpha}\iota_{V}\alpha
\wedge\gamma\right\}  _{\gamma,X+V}  & =\iota_{X+V}d\gamma\wedge\left(
\alpha+\left(  -1\right)  ^{\deg\alpha}\iota_{V}\alpha\wedge\gamma\right) \\
& -\gamma\wedge\iota_{X+V}d\left(  \alpha+\left(  -1\right)  ^{\deg\alpha
}\iota_{V}\alpha\wedge\gamma\right) \\
& =\iota_{X}d\gamma\wedge\alpha+\left(  -1\right)  ^{\deg\alpha}\iota
_{X}d\gamma\wedge\iota_{V}\alpha\wedge\gamma\\
& +\iota_{V}d\gamma\wedge\alpha+\left(  -1\right)  ^{\deg\alpha}\iota
_{V}d\gamma\wedge\iota_{V}\alpha\wedge\gamma\\
& -\gamma\wedge\iota_{X}d\alpha-\left(  -1\right)  ^{\deg\alpha}\gamma
\wedge\iota_{X}d\left(  \iota_{V}\alpha\wedge\gamma\right) \\
& -\gamma\wedge\iota_{V}d\alpha-\left(  -1\right)  ^{\deg\alpha}\gamma
\wedge\iota_{V}d\left(  \iota_{V}\alpha\wedge\gamma\right)
\end{align*}
and%
\[
d\left(  \alpha+\left(  -1\right)  ^{\deg\alpha}\iota_{V}\alpha\wedge
\gamma\right)  =d\alpha+\left(  -1\right)  ^{\deg\alpha}d\iota_{V}\alpha
\wedge\gamma-\iota_{V}\alpha\wedge d\gamma.
\]
So
\begin{align}
\delta_{\gamma,X+V}\Theta\alpha & =d\alpha+\left(  -1\right)  ^{\deg\alpha
}d\iota_{V}\alpha\wedge\gamma-\iota_{V}\alpha\wedge d\gamma\nonumber\\
& +\iota_{X}d\gamma\wedge\alpha+\left(  -1\right)  ^{\deg\alpha}\iota
_{X}d\gamma\wedge\iota_{V}\alpha\wedge\gamma+\iota_{V}d\gamma\wedge
\alpha\nonumber\\
& +\left(  -1\right)  ^{\deg\alpha}\iota_{V}d\gamma\wedge\iota_{V}\alpha
\wedge\gamma-\gamma\wedge\iota_{X}d\alpha-\left(  -1\right)  ^{\deg\alpha
}\gamma\wedge\iota_{X}d\left(  \iota_{V}\alpha\wedge\gamma\right) \nonumber\\
& -\gamma\wedge\iota_{V}d\alpha-\left(  -1\right)  ^{\deg\alpha}\gamma
\wedge\iota_{V}d\left(  \iota_{V}\alpha\wedge\gamma\right)  .\label{deltafi}%
\end{align}
Since%
\begin{align*}
\gamma\wedge\iota_{X}d\left(  \iota_{V}\alpha\wedge\gamma\right)   &
=\gamma\wedge\iota_{X}\left(  d\iota_{V}\alpha\wedge\gamma+\left(  -1\right)
^{\deg\alpha-1}\iota_{V}\alpha\wedge d\gamma\right) \\
& =\left(  -1\right)  ^{\deg\alpha}\gamma\wedge d\iota_{V}\alpha+\gamma
\wedge\iota_{V}\alpha\wedge\iota_{X}d\gamma\\
& =\left(  -1\right)  ^{\deg\alpha}\left(  \gamma\wedge d\iota_{V}\alpha
-\iota_{V}\alpha\wedge d\gamma\right)
\end{align*}
and%
\begin{align*}
\gamma\wedge\iota_{V}d\left(  \iota_{V}\alpha\wedge\gamma\right)   &
=\gamma\wedge\iota_{V}\left(  d\iota_{V}\alpha\wedge\gamma+\left(  -1\right)
^{\deg\alpha-1}\iota_{V}\alpha\wedge d\gamma\right) \\
& =\gamma\wedge\iota_{V}\alpha\wedge\iota_{V}d\gamma,
\end{align*}
(\ref{deltafi}) gives%
\begin{align}
\delta_{\gamma,X+V}\Theta\alpha & =\delta_{\gamma,X}\alpha+\left(  -1\right)
^{\deg\alpha}d\iota_{V}\alpha\wedge\gamma-\iota_{V}\alpha\wedge d\gamma
\nonumber\\
& +\left(  -1\right)  ^{\deg\alpha}\iota_{X}d\gamma\wedge\iota_{V}\alpha
\wedge\gamma+\iota_{V}d\gamma\wedge\alpha\nonumber\\
& +\left(  -1\right)  ^{\deg\alpha}\iota_{V}d\gamma\wedge\iota_{V}\alpha
\wedge\gamma-\gamma\wedge d\iota_{V}\alpha+\iota_{V}\alpha\wedge
d\gamma\nonumber\\
& -\gamma\wedge\iota_{V}d\alpha-\left(  -1\right)  ^{\deg\alpha}\gamma
\wedge\iota_{V}\alpha\wedge\iota_{V}d\gamma\nonumber\\
& =\delta_{\gamma,X}\alpha+d\gamma\wedge\iota_{V}\alpha+\iota_{V}d\gamma
\wedge\alpha-\gamma\wedge\iota_{V}d\alpha.\label{deltafi2}%
\end{align}
Finally, from (\ref{fidelta}) and (\ref{deltafi2}) it follows that%
\begin{align*}
\delta_{\gamma,X+V}\Theta\alpha-\Theta\delta_{\gamma,X}\alpha & =\iota
_{V}d\gamma\wedge\alpha+\left(  -1\right)  ^{\deg\alpha}\left(  \iota_{V}%
\iota_{X}d\gamma\right)  \alpha\wedge\gamma\\
& =-\iota_{V}\left(  \iota_{X}d\gamma\wedge\gamma\right)  \wedge\alpha+\left(
-1\right)  ^{\deg\alpha}\left(  \iota_{V}\iota_{X}d\gamma\right)  \alpha
\wedge\gamma\\
& =-\iota_{V}\left(  \iota_{X}d\gamma\right)  \gamma\wedge\alpha+\left(
-1\right)  ^{\deg\alpha}\left(  \iota_{V}\iota_{X}d\gamma\right)  \alpha
\wedge\gamma=0.
\end{align*}

\end{proof}

\subsection{Moduli space of deformations of integrable distributions of
codimension $1$}

\bigskip
\ \ \ \ \ \ \ \ \ \ \ \ \ \ \ \ \ \ \ \ \ \ \ \ \ \ \ \ \ \ \ \ \ \ \ \ \ \ \ \ \ \ \ \ \newline%

Let $L$ be a $C^{\infty}$ manifold and $\xi\subset T\left(  L\right)  $ an
integrable co-orientable distribution of codimension $1$. We fix a DGLA
defining couple $\left(  \gamma,X\right)  $ and we consider the DGLA $\left(
\mathcal{Z}^{\ast}\left(  L\right)  ,\delta,\left\{  \cdot,\cdot\right\}
\right)  $ previously defined.

\begin{lemma}
\label{Ker(gama+alfa) integrable}Let $\alpha\in\mathcal{Z}^{1}\left(
L\right)  $. The following are equivalent:

i) The distribution $\xi_{\alpha}=\ker$~$\left(  \gamma+\alpha\right)  $ is integrable.

ii) $\alpha$ satisfies the Maurer-Cartan equation (\ref{MC}) in \ $\left(
\mathcal{Z}^{\ast}\left(  L\right)  ,\delta,\left\{  \cdot,\cdot\right\}
\right)  $.
\end{lemma}

\begin{proof}
By Lemma \ref{Frobenius} the distribution $\ker$~$\left(  \gamma
+\alpha\right)  $ is integrable if and only if $\gamma+\alpha$ satisfies
(\ref{MC}) in $\left(  \Lambda^{\ast}\left(  L\right)  ,d,\left\{  \cdot
,\cdot\right\}  \right)  $. Since $\gamma$ satisfies (\ref{MC}) we have
\begin{align*}
d\left(  \gamma+\alpha\right)  +\frac{1}{2}\left\{  \gamma+\alpha
,\gamma+\alpha\right\}   & =d\alpha+\left\{  \gamma,\alpha\right\}  +\frac
{1}{2}\left\{  \alpha,\alpha\right\} \\
& =\delta\alpha+\frac{1}{2}\left\{  \alpha,\alpha\right\}
\end{align*}
and the Lemma follows.
\end{proof}

\begin{notation}%
\[
\mathfrak{MC}_{\delta}\left(  L\right)  =\left\{  \alpha\in\mathcal{Z}%
^{1}\left(  L\right)  :\ \delta a+\frac{1}{2}\left\{  \alpha,\alpha\right\}
=0\right\}  .
\]

\end{notation}

Following \cite{Kodaira61} we define:

\begin{definition}
By a differentiable family of deformations of an integrable distribution $\xi$
we mean a differentiable family $\omega:\mathcal{D}=\left(  \xi_{t}\right)
_{t\in I}\mapsto t\in I=]-a,a[$, $a>0$, \ of integrable distributions such
that $\xi_{0}=\omega^{-1}\left(  0\right)  =\xi$. \ By a differentiable family
of small deformations of an integrable distribution $\xi$ \ we mean the
restriction $\mathcal{D}\left\vert I_{\varepsilon}\right.  =\omega^{-1}\left(
I_{\varepsilon}\right)  $ of a differentiable family of $\omega:\mathcal{D}%
\rightarrow I_{\varepsilon}=]-\varepsilon,\varepsilon\lbrack$ of deformations
of $\xi=\omega^{-1}\left(  0\right)  $ to a sufficiently small neighborhood of
$0$ in $I$.
\end{definition}

\begin{remark}
By Lemma \ref{Ker(gama+alfa) integrable} a differentiable family of
deformations of an integrable distribution is given by a differentiable family
$\left(  \alpha_{t}\right)  _{t\in I}$ in $\mathcal{Z}^{1}\left(  L\right)  $
such that $\xi_{t}=\ker\alpha_{t}$ and $\alpha_{0}=0$.
\end{remark}

\begin{definition}
\label{Def group action} Let $\mathcal{U}$ be a neighborhood of the identity
in $\mathcal{G}$ and $\mathcal{V}$ be a neighborhood of $0$ in $\mathcal{Z}%
^{1}\left(  L\right)  $ such that $\Phi^{\ast}\left(  \gamma+\alpha\right)
\left(  X\right)  \neq0$, $\left(  \Phi^{-1}\right)  ^{\ast}\left(
\gamma+\alpha\right)  \left(  X\right)  \neq0$ for every $\left(  \Phi
,\alpha\right)  \in\mathcal{U}\times\mathcal{V}$. We define
\begin{equation}
\left(  \Phi,\alpha\right)  \in\mathcal{U}\times\mathcal{V}\subset
\mathcal{G}\times\mathcal{Z}^{1}\left(  L\right)  \rightarrow\mathcal{Z}%
^{1}\left(  L\right)  \ni\chi\left(  \Phi\right)  \left(  \alpha\right)
=\left(  \Phi^{\ast}\left(  \gamma+\alpha\right)  \left(  X\right)  \right)
^{-1}\Phi^{\ast}\left(  \gamma+\alpha\right)  -\gamma.\label{chi(fi)}%
\end{equation}

\end{definition}

\begin{remark}
\label{group action} The previous definition is adapted for small
deformations. If $\beta=\chi\left(  \Phi\right)  \left(  \alpha\right)  $,
$\xi_{\chi\left(  \Phi\right)  \left(  \alpha\right)  }=\Phi^{\ast}\xi
_{\alpha}$. This means that $\xi_{\alpha}$ is integrable if and only if
$\xi_{\chi\left(  \Phi\right)  \left(  \alpha\right)  }$ is integrable. By
Lemma \ref{Ker(gama+alfa) integrable} we deduce that $\alpha$ satisfies the
Maurer-Cartan equation (\ref{MC}) in the DGLA\ $\left(  \mathcal{Z}^{\ast
}\left(  L\right)  ,\delta,\left\{  \cdot,\cdot\right\}  \right)  $ if and
only if $\chi\left(  \Phi\right)  \left(  \alpha\right)  $ does.
\end{remark}

\begin{remark}
We consider the right action of the group $\mathcal{G=}Diff\left(  L\right)  $
of diffeomorphisms of $L$ on the set $\mathcal{D}$ of distributions of
codimension $1$ on $L$ given by%
\begin{equation}
\tau\left(  \Phi\right)  \left(  \xi\right)  =\Phi_{\ast}^{-1}\xi,\ \Phi
\in\mathcal{G},\ \xi\in\mathcal{D}.\label{Tau}%
\end{equation}
Denote by $\mathcal{I}$ the subset of $\mathcal{D}$ given by the coorientable
integrable distributions. Since $\xi=\ker\beta$ if and only if $\tau\left(
\Phi\right)  \left(  \xi\right)  =\ker\Phi^{\ast}\beta$ it follows that
$\mathcal{I}$ is $\mathcal{G}$-invariant.
\end{remark}

\begin{definition}
\label{def moduli foliations}i) $\mathcal{I}/\mathcal{G}$ is the moduli space
of integrable distributions of codimension $1$ on $L$.

ii) We consider the one-to-one mapping
\begin{equation}
\mathcal{Z}^{1}\left(  L\right)  \ni\alpha\mapsto\zeta_{\alpha}=\ker\left(
\gamma+\alpha\right)  \in\mathcal{R},\label{R}%
\end{equation}
where $\mathcal{R=}\left\{  \zeta\in\mathcal{D}:\ \zeta=\ker\left(
\gamma+\beta\right)  ,\ \beta\in\mathcal{Z}^{1}\left(  L\right)  \right\}
\subset\mathcal{D}$. The moduli space of deformations of integrable
distributions of codimension $1$ of $\xi$ is $\pi^{-1}\left(  \pi\left(
\mathcal{I\cap R}\right)  \right)  /\mathcal{G}$, where $\pi:\mathcal{D}%
\rightarrow\mathcal{D}/\mathcal{I}$ is the canonical map.
\end{definition}

\begin{remark}
\label{local action}Let $\nu\in$ $\pi^{-1}\left(  \pi\left(  \mathcal{I\cap
R}\right)  \right)  /\mathcal{G}$, $\nu=\pi\left(  \zeta\right)  $, where
$\zeta\in\mathcal{I\cap R}$. By Lemma \ref{Ker(gama+alfa) integrable} there
exists $\alpha\in\mathfrak{MC}_{\delta}\left(  L\right)  $ such that
$\zeta=\zeta_{\alpha}=\ker\left(  \gamma+\alpha\right)  $. Then if $\Phi
\in\mathcal{G}$ is sufficiently close to the identity we have
\[
\tau\left(  \Phi\right)  \left(  \zeta_{\alpha}\right)  =\Phi_{\ast}^{-1}%
\zeta_{\alpha}=\ker\Phi^{\ast}\left(  \gamma+\alpha\right)  =\ker\left(
\gamma+\chi\left(  \Phi\right)  \left(  \alpha\right)  \right)  =\zeta
_{\chi\left(  \Phi\right)  \left(  \alpha\right)  },
\]
so $\nu=\pi\left(  \zeta_{\chi\left(  \Phi\right)  \left(  \alpha\right)
}\right)  $ and the action given by (\ref{chi(fi)}) is the local description
of the global action given by (\ref{Tau}) via the correspondence (\ref{R}).
\end{remark}

\begin{notation}
We will denote the moduli space of deformations of integrable distributions of
codimension $1$ of $\xi$ by $\mathfrak{MC}_{\delta}\left(  L\right)
/\thicksim_{\mathcal{G}}$.
\end{notation}

\begin{remark}
\label{DGM}Let $\mathcal{G}^{0}$ be the identity component of $\mathcal{G}$,
$\Lambda^{1}\left(  L\right)  ^{\prime}$ the set of nowhere vanishing
$1$-forms on $L$ and $\Lambda^{1}\left(  L\right)  ^{\prime}/e^{\Lambda
^{0}\left(  L\right)  }$ the set of cooriented distributions. Then we have the
group action
\[
\mathcal{G}^{0}\times\Lambda^{1}\left(  L\right)  ^{\prime}/e^{\Lambda
^{0}\left(  L\right)  }\ni\left(  \Phi,\ker\gamma_{\alpha}\right)
\rightarrow\ker\chi\left(  \Phi\right)  \left(  \alpha\right)  \in\Lambda
^{1}\left(  L\right)  ^{\prime}/e^{\Lambda^{0}\left(  L\right)  }%
\]
of $\mathcal{G}^{0}$ on $\Lambda^{1}\left(  L\right)  ^{\prime}/e^{\Lambda
^{0}\left(  L\right)  }$ and consider $\left[  \mathfrak{MC}_{\delta}\left(
L\right)  /\mathcal{G}^{0}\right]  $ the associated transformation groupoid
(see \cite{Goldman88} for the definition of transformation groupoids). Another
possibility of defining $\mathfrak{MC}_{\delta}\left(  L\right)
/\thicksim_{\mathcal{G}}$is to take the germ at $\left(  Id_{L},\xi\right)  $.
\end{remark}

The moduli space of deformations of integrable distributions of codimension
$1$ depends a priori on the DGLA defining couple. We will now prove that the
moduli space $\mathfrak{MC}_{\delta_{\gamma,X}}\left(  L\right)
/\thicksim_{\mathcal{G}}$ and $\mathfrak{MC}_{\delta_{\widehat{\gamma
},\widehat{X}}}\left(  L\right)  /\thicksim_{\mathcal{G}}$ of deformations of
integrable distributions of codimension $1$ corresponding to defining couples
$\left(  \gamma,X\right)  $ and $\left(  \widehat{\gamma},\widehat{X}\right)
$ are canonically isomorphic:

\begin{proposition}
\label{Iso MC}Let $L$ be a $C^{\infty}$ manifold and $\xi\subset T\left(
L\right)  $ an integrable distribution of codimension $1$. Let $\left(
\gamma,X\right)  $ be a DGLA defining couple and $V\neq0$ a $\xi$-valued
vector field on $L$. Let $\mathfrak{U}_{V}=\left\{  \alpha\in Z^{1}\left(
L\right)  :\left(  1+\iota_{V}\alpha\right)  \left(  x\right)  \neq0,x\in
L\right\}  $. For $\alpha\in$ $\mathfrak{U}_{V}$ define $F_{V}\alpha=\left(
1+\iota_{V}\alpha\right)  ^{-1}\left(  \alpha-\left(  \iota_{V}\alpha\right)
\gamma\right)  $. Then $F_{V}:\ \mathfrak{MC}_{\delta_{\gamma,X}}\left(
L\right)  \cap\mathfrak{U}_{V}\rightarrow\mathfrak{MC}_{\delta_{\gamma,X+V}%
}\left(  L\right)  \cap\mathfrak{U}_{V}$ is an isomorphism which induces an
isomorphism $\widetilde{F_{V}}:\ \mathfrak{MC}_{\delta_{\gamma,X}}\left(
L\right)  \cap\mathfrak{U}_{V}/\thicksim_{\mathcal{G}}\rightarrow
\mathfrak{MC}_{\delta_{\gamma,X+V}}\left(  L\right)  \cap\mathfrak{U}%
_{V}/\thicksim_{\mathcal{G}}$
\end{proposition}

\begin{proof}
Let $\alpha\in\ \mathfrak{MC}_{\delta_{\gamma,X}}\left(  L\right)
\cap\mathfrak{U}_{V}$. The Lemma \ref{Ker(gama+alfa) integrable} implies that
$Ker\left(  \gamma+\alpha\right)  $ is integrable. Since
\[
\iota_{X+V}F_{V}\alpha=\left(  1+\iota_{V}\alpha\right)  ^{-1}\iota
_{X+V}\left(  \alpha-\left(  \iota_{V}\alpha\right)  \gamma\right)  =\left(
1+\iota_{V}\alpha\right)  ^{-1}\left(  \iota_{V}\alpha-\iota_{V}\alpha\right)
=0,
\]
it follows that $F_{V}\alpha\in Z_{\gamma,X+V}^{1}\left(  L\right)  $. From
Proposition \ref{Iso DGVS} it follows that $F_{V}$ is the restriction to
$\mathfrak{MC}_{\delta_{\gamma,X}}\left(  L\right)  \cap\mathfrak{U}_{V}$ of
the DGVS-isomorphism $\left(  1+\iota_{V}\alpha\right)  ^{-1}\Theta_{V}$,
where $\Theta_{V}$ was defined in Proposition \ref{Iso DGVS}. We have
\[
\left(  \gamma+F_{V}\alpha\right)  =\left(  \gamma+\left(  1+\iota_{V}%
\alpha\right)  ^{-1}\left(  \alpha-\left(  \iota_{V}\alpha\right)
\gamma\right)  \right)  =\gamma+\alpha,
\]
so $Ker\left(  \gamma+\alpha\right)  =Ker\left(  \gamma+F_{V}\alpha\right)  $
and by using again the Lemma \ref{Ker(gama+alfa) integrable} we obtain
$F_{V}\alpha\in\ \mathfrak{MC}_{\delta_{\gamma,X+V}}\left(  L\right)  $.

The invariance of $F_{V}$ follows by Remark \ref{group action}.
\end{proof}

From Proposition \ref{Iso DGVS} and Proposition \ref{Iso MC} we obtain

\begin{corollary}
Let $L$ be a $C^{\infty}$ manifold and $\xi\subset T\left(  L\right)  $ an
integrable distribution of codimension $1$. Let $\left(  \gamma,X\right)  $,
$\left(  \widehat{\gamma},\widehat{X}\right)  $ be DGLA defining couples,
$\widehat{\gamma}=\pm e^{\lambda}\gamma$, $\widehat{X}=\pm e^{-\lambda}X+V $
with $\lambda\in C^{\infty}\left(  L\right)  $ and $V$ a $\xi$-valued vector
field. Then there exists a canonical isomorphism $F:\mathfrak{MC}%
_{\delta_{\gamma,X}}\left(  L\right)  /\thicksim_{\mathcal{G}}\rightarrow
\mathfrak{MC}_{\delta_{\widehat{\gamma},\widehat{X}}}\left(  L\right)
/\thicksim_{\mathcal{G}}$ between the the moduli space of deformations of
integrable distributions of codimension $1$ of $\xi$, $F=\Theta_{V}\circ
\Psi_{\lambda}$, $\Psi_{\lambda}:\mathcal{Z}_{\gamma,X}^{\ast}\left(
L\right)  \rightarrow\mathcal{Z}_{\widehat{\gamma},e^{-\lambda}X}^{\ast
}\left(  L\right)  $, $\Theta_{V}:\mathcal{Z}_{\widehat{\gamma},e^{-\lambda}%
X}^{\ast}\left(  L\right)  \rightarrow\mathcal{Z}_{\left(  \widehat{\gamma
},\widehat{X}\right)  }^{\ast}\left(  L\right)  $, $\Psi_{\lambda}\left(
\alpha\right)  =e^{\lambda}\alpha$ and $\Theta\left(  \alpha\right)
=\Theta_{V}\left(  \alpha\right)  =\alpha+\left(  -1\right)  ^{\deg\alpha
}\iota_{V}\alpha\wedge\gamma$.
\end{corollary}

\begin{lemma}
\label{d/dt(chi)=delta} Let $Y$ be a vector field on $L$ and $\Phi^{Y}$ the
flow of $Y$. Then%
\[
\frac{d\chi\left(  \Phi_{t}^{Y}\right)  }{dt}_{\left\vert t=0\right.  }\left(
0\right)  =-\delta\left(  \iota_{Y}\gamma\right)  .
\]

\end{lemma}

\begin{proof}
We have
\begin{align*}
\frac{d\chi\left(  \Phi_{t}^{Y}\right)  }{dt}_{\left\vert t=0\right.  }\left(
0\right)   & =\frac{d\left(  \left(  \left(  \left(  \Phi_{t}^{Y}\right)
^{-1}\right)  ^{\ast}\left(  \gamma\right)  \left(  X\right)  \right)
^{-1}\left(  \left(  \Phi_{t}^{Y}\right)  ^{-1}\right)  ^{\ast}\left(
\gamma\right)  -\gamma\right)  }{dt}_{\left\vert t=0\right.  }\\
& =\left(  \left(  \Phi_{t}^{Y}\right)  ^{-1}\right)  ^{\ast}\left(
\gamma\right)  \frac{d\left(  \left(  \left(  \Phi_{t}^{Y}\right)
^{-1}\right)  ^{\ast}\left(  \gamma\right)  \left(  X\right)  ^{-1}\right)
}{dt}_{\left\vert t=0\right.  }\left(  0\right) \\
& +\left(  \left(  \Phi_{t}^{Y}\right)  ^{-1}\right)  ^{\ast}\left(
\gamma\right)  \left(  X\right)  ^{-1}\frac{d\left(  \left(  \Phi_{t}%
^{Y}\right)  ^{-1}\right)  ^{\ast}}{dt}_{\left\vert t=0\right.  }\\
& =\frac{d\left(  \left(  \left(  \left(  \Phi_{t}^{Y}\right)  ^{-1}\right)
^{\ast}\left(  \gamma\right)  \left(  X\right)  \right)  ^{-1}\right)  }%
{dt}_{\left\vert t=0\right.  }\gamma+\frac{d\left(  \left(  \left(  \Phi
_{t}^{Y}\right)  ^{-1}\right)  ^{\ast}\left(  \gamma\right)  \right)  }%
{dt}_{\left\vert t=0\right.  }\\
& =\mathcal{L}_{Y}\left(  \gamma\right)  \left(  X\right)  \gamma
-\mathcal{L}_{Y}\gamma\\
& =\left(  d\iota_{Y}\gamma\right)  \left(  X\right)  \gamma+\iota_{Y}%
d\gamma\left(  X\right)  \gamma-d\iota_{Y}\gamma-\iota_{Y}d\gamma.
\end{align*}
By Lemma \ref{Frobenius} iv)%
\begin{align*}
\iota_{Y}d\gamma & =-\iota_{Y}\left(  \iota_{X}d\gamma\wedge\gamma\right)
=-\left(  \iota_{Y}\left(  \iota_{X}d\gamma\right)  \right)  \gamma+\left(
\iota_{Y}\gamma\right)  \iota_{X}d\gamma\\
& =-\left(  d\gamma\left(  X,Y\right)  \right)  \gamma+\left(  \iota_{Y}%
\gamma\right)  \iota_{X}d\gamma,
\end{align*}
so%
\begin{align}
\frac{d\chi\left(  \Phi_{t}^{Y}\right)  }{dt}_{\left\vert t=0\right.  }\left(
0\right)   & =\left(  d\iota_{Y}\gamma\right)  \left(  X\right)
\gamma-d\gamma\left(  Y,X\right)  \gamma-d\iota_{Y}\gamma\nonumber\\
& +\left(  d\gamma\left(  X,Y\right)  \right)  \gamma-\left(  \iota_{Y}%
\gamma\right)  \iota_{X}d\gamma\nonumber\\
& =\left(  \iota_{X}d\iota_{Y}\gamma\right)  \gamma-d\iota_{Y}\gamma-\left(
\iota_{Y}\gamma\right)  \iota_{X}d\gamma.\label{A}%
\end{align}
Since
\[
\mathcal{L}_{X}\gamma=d\iota_{X}\gamma+\iota_{X}d\gamma=\iota_{X}d\gamma
\]
it follows that%
\begin{align}
\delta\iota_{Y}\gamma & =d\iota_{Y}\gamma+\left\{  \gamma,\iota_{Y}%
\gamma\right\}  =d\iota_{Y}\gamma+\mathcal{L}_{X}\gamma\wedge\iota_{Y}%
\gamma-\gamma\wedge\mathcal{L}_{X}\iota_{Y}\gamma\label{B}\\
& =d\iota_{Y}\gamma+\left(  \iota_{Y}\gamma\right)  \iota_{X}d\gamma-X\left(
\iota_{Y}\gamma\right)  \gamma.\nonumber
\end{align}
From (\ref{A}) and (\ref{B}) we obtain
\[
\frac{d\chi\left(  \Phi_{t}^{Y}\right)  }{dt}_{\left\vert t=0\right.  }\left(
0\right)  =-\delta\iota_{Y}\gamma.
\]

\end{proof}

\begin{definition}
A $\ \mathfrak{MC}_{\delta}\left(  L\right)  $-valued curve through the origin
is a continuous mapping $\lambda:\left[  -a,a\right]  \rightarrow
\mathfrak{MC}_{\delta}\left(  L\right)  $, $a>0$, such that $\lambda\left(
0\right)  =0$. We say that $\alpha$ is the tangent vector at the origin of the
$\mathfrak{MC}_{\delta}\left(  L\right)  $-valued curve $\lambda$ through the
origin to $\mathfrak{MC}_{\delta}\left(  L\right)  $ if $\alpha=\underset
{t\rightarrow0}{\lim}\frac{\lambda\left(  t\right)  }{t}=\frac{d\lambda}%
{dt}_{\left\vert t=0\right.  }$.
\end{definition}

\begin{proposition}
\label{Tangent Inclus Cohomologie}Let $\alpha$ be the tangent vector at the
origin of a $\mathfrak{MC}_{\delta}\left(  L\right)  $-valued curve through
the origin $\lambda$, $Y$ a vector field on $L$ and $\Phi^{Y}$ the flow of
$Y$. Set $\mu\left(  t\right)  =\chi\left(  \Phi_{t}^{Y}\right)  \left(
\lambda\left(  t\right)  \right)  $., Then:

i) $\delta\alpha=0$.

ii) The tangent vector $\beta$ at the origin of \ the $\mathfrak{MC}_{\delta
}\left(  L\right)  $-valued curve $\mu$ is%
\[
\beta=\alpha-\delta\iota_{Y}\gamma.
\]

\end{proposition}

\begin{proof}
i) By Lemma \ref{Ker(gama+alfa) integrable} $\lambda\left(  t\right)  $
verifies the Maurer Cartan equation for every $t$. Since $\lambda\left(
t\right)  =\alpha t+o\left(  t\right)  $, we have $\delta\alpha=0$.

ii)
\[
\beta=\frac{d\mu}{dt}_{\left\vert t=0\right.  }=\frac{d}{dt}\chi\left(
\Phi_{t}^{Y}\left(  \lambda\left(  t\right)  \right)  \right)  _{\left\vert
t=0\right.  }=\frac{d\chi\left(  \Phi_{t}^{Y}\right)  }{dt}_{\left\vert
t=0\right.  }\left(  0\right)  +\alpha.
\]
The Proposition \ref{Tangent Inclus Cohomologie} follows now by Lemma
\ref{d/dt(chi)=delta}.
\end{proof}

The Proposition \ref{Tangent Inclus Cohomologie} justifies the following definition:

\begin{definition}
The tangent cone $T_{\left[  0\right]  }\left(  \mathfrak{MC}_{\delta}\left(
L\right)  /\thicksim_{\mathcal{G}}\right)  $ at $\left[  0\right]  $ to
$\mathfrak{MC}_{\delta}\left(  L\right)  /\thicksim_{\mathcal{G}}$ is the
collection of cohomology classes in $H^{1}\left(  \mathcal{Z}\left(  L\right)
,\delta\right)  $ of the tangent vectors at $0$ to $\mathfrak{MC}_{\delta
}\left(  L\right)  $-valued curves.
\end{definition}

\begin{definition}
We say that the deformation theory is not obstructed at $\left[  0\right]
$\ if
\[
T_{\left[  0\right]  }\left(  \mathfrak{MC}_{\delta}\left(  L\right)
/\thicksim_{\mathcal{G}}\right)  =H^{1}\left(  \mathcal{Z}\left(  L\right)
,\delta\right)  .
\]

\end{definition}

\begin{remark}
In general, to establish unobstructedness of a deformation theory is a very
hard problem and conditions as the vanishing of%
\[
q:H^{1}\left(  \mathcal{Z}\left(  L\right)  ,\delta\right)  \rightarrow
H^{2}\left(  \mathcal{Z}\left(  L\right)  ,\delta\right)  ,\ q\left(
a\right)  =\left\{  a,a\right\}  ,
\]
will provide only curves of formal solutions to the Maurer-Cartan equation
with prescribed tangent vectors at 0 (see for ex. \cite{Bartolomeis05}).
\end{remark}

\begin{remark}
\label{iso Z(L) Lambda(L)}There exists a natural isomorphism $\Theta
:\Lambda^{\ast}\left(  \xi\right)  \rightarrow\mathcal{Z}^{\ast}\left(
L\right)  $: for $\alpha\in\Lambda^{1}\left(  \xi\right)  $ set $\Theta\left(
\alpha\right)  \left(  X\right)  =0$, $\Theta\left(  \alpha\right)  \left(
Y\right)  =\alpha\left(  Y\right)  $ if $Y\in\xi$ and extend by linearity. Let
$d_{b}:\Lambda^{\ast}\left(  \xi\right)  \rightarrow\Lambda^{\ast}\left(
\xi\right)  $ be the differential along the leaves of $\xi$. By using this
isomorphism we consider $d_{b}:\mathcal{Z}^{\ast}\left(  L\right)
\rightarrow\mathcal{Z}^{\ast}\left(  L\right)  $ and for every $\alpha
\in\mathcal{Z}^{\ast}\left(  L\right)  $ we have%
\begin{equation}
d_{b}\alpha=\iota_{X}\left(  \gamma\wedge d\alpha\right)  =d\alpha
-\gamma\wedge\iota_{X}d\alpha.\label{db}%
\end{equation}
Indeed let $\alpha\in\Lambda^{p}\left(  \xi\right)  $ and $X_{1},\cdot
\cdot\cdot,X_{p+1}\in\xi$. Since $\gamma\left(  X_{j}\right)  =0$,
$j=1,\cdot\cdot\cdot,p+1$ and $\gamma\left(  X\right)  =1$, we have
\[
\iota_{X}\left(  \gamma\wedge d\alpha\right)  \left(  X_{1},\cdot\cdot
\cdot,X_{p+1}\right)  =\left(  \gamma\wedge d\alpha\right)  \left(
X,X_{1},\cdot\cdot\cdot,X_{p+1}\right)  =d\alpha\left(  X_{1},\cdot\cdot
\cdot,X_{p+1}\right)  .
\]

\end{remark}

\begin{lemma}
\label{i(X) d(gama ) db closed}The form $\iota_{X}d\gamma$ is $d_{b}$-closed.
\end{lemma}

\begin{proof}
From Lemma \ref{Frobenius} $iii)$ we obtain%
\[
0=d\left(  \gamma\wedge\iota_{X}d\gamma\right)  =d\gamma\wedge\iota_{X}%
d\gamma-\gamma\wedge d\iota_{X}d\gamma=-\gamma\wedge d\iota_{X}d\gamma
\]
so $\iota_{X}\left(  \gamma\wedge d\iota_{X}d\gamma\right)  =0$ and the Lemma
follows by (\ref{db}) .
\end{proof}

\begin{notation}
The cohomology class $\left[  \iota_{X}d\gamma\right]  \in H^{1}\left(
\Lambda^{\ast}\left(  \xi\right)  ,d_{b}\right)  $ which depends only on $\xi$
will be denoted by \bigskip$c\left(  \xi\right)  $.
\end{notation}

\begin{lemma}
\label{relation delta b delta}Let $\alpha\in\mathcal{Z}^{p}\left(  L\right)
$. Then
\begin{equation}
\delta\alpha=d_{b}\alpha+\iota_{X}d\gamma\wedge\alpha.\label{rel delta db}%
\end{equation}
In particular
\[
d_{b}\alpha=\delta\alpha\iff\iota_{X}d\gamma\wedge\alpha=0.
\]

\end{lemma}

\begin{proof}
By (\ref{db}) and (\ref{gama alfa}) we have%
\[
\delta\alpha=d\alpha+\left\{  \gamma,\alpha\right\}  =d\alpha+\iota_{X}%
d\gamma\wedge\alpha-\gamma\wedge\iota_{X}d\alpha=d_{b}\alpha+\iota_{X}%
d\gamma\wedge\alpha
\]
and the lemma follows.
\end{proof}

\begin{remark}
\label{Comparaison avec Kodaira Spencer}We would like to mention that Kodaira
and Spencer developped in \cite{Kodaira61} a theory of deformations of the so
called multifoliate structures, which are more general then the foliate
structures. A multifoliate structure on an orientable manifold $X$ \ of
dimension $n$ is an atlas $\left(  U_{i},\left(  x_{i}^{\alpha}\right)
_{\alpha=1,\cdot\cdot\cdot,n}\right)  $such that the changes of coordinates
verify%
\[
\frac{\partial x_{i}^{\alpha}}{\partial x_{k}^{\beta}}=0\ for\ \beta
\nsucceq\alpha,
\]
where $\left(  \mathcal{P},\geqq\right)  $ is a finite partially ordered set,
$\left\{  \alpha\right\}  $ a set of integers such there is given a map
$\left\{  \alpha\right\}  \mapsto\left[  \alpha\right]  $ of $\alpha$ onto
$\mathcal{P}$ and the order relation $"\succapprox"$ is defined by
$\alpha>\beta$ if and only if $\left[  \alpha\right]  >\left[  \beta\right]
$, $\alpha\thicksim\beta$ if and only if $\left[  \alpha\right]  =\left[
\beta\right]  $. An usual foliation is the particular case when $\mathcal{P}%
=\left\{  a,b\right\}  $, $a>b$.

Kodaira and Spencer define in \cite{Kodaira61} subsheafs $\Phi_{\mathcal{P}%
}^{p}$, $p\in\mathbb{N}$, of the sheaf of germs of jet forms of degree $p$ on
$X$ which are compatible with the multifoliate structure and a differential
$D$ such that%
\[
0\rightarrow\Theta_{\mathcal{P}}\overset{D}{\rightarrow}\Phi_{\mathcal{P}}%
^{1}\overset{D}{\rightarrow}\Phi_{\mathcal{P}}^{2}\overset{D}{\rightarrow
}\cdot\cdot\cdot\overset{D}{\rightarrow}\Phi_{\mathcal{P}}^{n}\rightarrow0
\]
is a resolution of the sheaf $\Theta_{\mathcal{P}}$ of the vector fields
tangent to the multifoliate structure. They define also a Lie bracket $\left[
\cdot,\cdot\right]  $ on jet forms such that $\left(  \left(  \oplus_{p=1}%
^{n}\ker D\right)  \left(  X\right)  ,\ D,\left[  \cdot,\cdot\right]  \right)
$ is a DGLA and every small deformation of the multifoliate structure is given
by a family $\left\{  v\left(  t\right)  \right\}  \subset\Phi_{\mathcal{P}%
}^{1}\left(  X\right)  $ \ verifying $\left[  v\left(  t\right)  ,v\left(
t\right)  \right]  =0$ and $v\left(  0\right)  =d$. So $v\left(  t\right)  +d$
verifies the Maurer Cartan equation. Moreover $\frac{\partial v}{\partial
t}_{\left\vert t=0\right.  }\in Z\left(  \Phi_{\mathcal{P}}^{1}\right)  $ and
the class $\left[  \frac{\partial v}{\partial t}_{\left\vert t=0\right.
}\right]  \in H^{1}\left(  X,\Theta_{\mathcal{P}}\right)  $ represents the
infinitesimal deformation of the multifoliate structure along a tangent vector
$\frac{\partial}{\partial t}$.

In our approach, defined only for deformation of foliations of codimension $1
$, the DGLA algebra $\left(  \mathcal{Z}^{\ast}\left(  L\right)
,\delta,\left\{  \cdot,\cdot\right\}  \right)  $ associated to a foliation on
a cooriented manifold $L$ is a subalgebra of the the algebra $\left(
\Lambda^{\ast}\left(  L\right)  ,\delta,\left\{  \cdot,\cdot\right\}  \right)
$ of forms on $L$.\ Its definition depend on the choice of a DGLA defining
couple, but the cohomology class of this algebra does not depend on its
choice. The deformations are given by forms in $\mathcal{Z}^{1}\left(
L\right)  $ verifying the Maurer Cartan equations and the moduli space takes
in account the diffeomorphic deformations. The infinitesimal deformations
along curves are subsets of of the first cohomology group of the DGLA $\left(
\mathcal{Z}^{\ast}\left(  L\right)  ,\delta,\left\{  \cdot,\cdot\right\}
\right)  $.\ \ \ \ \ \ \ \ \newline
\end{remark}

\subsection{Transversally parallelizable foliations\bigskip}

\ \ \ \ \ \ \ \ \ \newline

Recall the following

\begin{definition}
Let $L$ be a $C^{\infty}$ manifold and $\xi\subset T\left(  L\right)  $ a
distribution of codimension $1$. $\xi$ is called transversally parallelizable
if there exists a $1$-form $\omega$ on $L$ such that $\xi=\ker\omega$ and
$d\omega=0$.
\end{definition}

\begin{proposition}
\label{transversally par}Let $L$ be a $C^{\infty}$ manifold and $\xi\subset
T\left(  L\right)  $ a distribution of codimension $1$ and $\left(
\gamma,X\right)  $ a DGLA defining couple. The following assertions are equivalent:

i) $\xi$ is transversally parallelizable.

ii)\ $c\left(  \xi\right)  =0$.

iii) There exists $\lambda\in C^{\infty}\left(  L\right)  $ such that
$\iota_{X}d\left(  e^{\lambda}\gamma\right)  =0$.

iv) There exists a DGLA defining couple $\left(  \widehat{\gamma},\widehat
{X}\right)  $ such that $\delta_{\widehat{\gamma},\widehat{X}}=d_{b}$.
\end{proposition}

\begin{proof}
The assertion i)$\implies$iii) is obvious and iii)$\iff$iv) by Lemma
\ref{relation delta b delta}.

$iv)\implies i)$ We may suppose that $\lambda\in C^{\infty}\left(  L\right)  $
such that $\widehat{\gamma}=e^{\lambda}\gamma$ and $\widehat{X}=e^{-\lambda
}X+V$, $V\in\xi$. The Lemma \ref{relation delta b delta} applied to $0$-forms
implies%
\[
\iota_{\widehat{X}}d\left(  e^{\lambda}\gamma\right)  =0
\]
and by Lemma \ref{Frobenius} $iv)$ it follows that%
\[
d\left(  e^{\lambda}\gamma\right)  =-\iota_{\widehat{X}}d\left(  e^{\lambda
}\gamma\right)  \wedge e^{\lambda}\gamma=0.
\]

$i)\implies ii)$ Let $\lambda\in C^{\infty}\left(  L\right)  $ such that
$d\left(  e^{\lambda}\gamma\right)  =0$. Since \
\[
d\left(  e^{\lambda}\gamma\right)  =e^{\lambda}\left(  d\gamma+d\lambda
\wedge\gamma\right)  =e^{\lambda}\left(  -\iota_{X}d\gamma\wedge
\gamma+d\lambda\wedge\gamma\right)  =0
\]
it follows that%
\begin{equation}
d\lambda\wedge\gamma=\iota_{X}d\gamma\wedge\gamma.\label{02}%
\end{equation}
We have%
\begin{equation}
\iota_{X}\left(  d\lambda\wedge\gamma\right)  =\left(  \iota_{X}%
d\lambda\right)  \gamma-d\lambda\label{03}%
\end{equation}
and
\begin{equation}
\iota_{X}\left(  \iota_{X}d\gamma\wedge\gamma\right)  =-\iota_{X}%
d\gamma,\label{04}%
\end{equation}
so by (\ref{02}), (\ref{03}) and (\ref{04}) we obtain
\begin{equation}
\left(  \iota_{X}d\lambda\right)  \gamma-d\lambda=-\iota_{X}d\gamma
.\label{004}%
\end{equation}
From (\ref{db}) and (\ref{004}) it follows that
\[
d_{b}\lambda=d\lambda-\left(  \iota_{X}d\lambda\right)  \gamma=\iota
_{X}d\gamma,
\]
so $c\left(  \xi\right)  =0$.

$ii)\implies i)$ Let $\lambda\in C^{\infty}\left(  L\right)  $ such that
\[
d_{b}\lambda=\iota_{X}d\gamma=d\lambda-\left(  \iota_{X}d\lambda\right)
\gamma.
\]
Then%
\begin{align*}
d\left(  e^{\lambda}\gamma\right)   & =e^{\lambda}\left(  d\gamma
+d\lambda\wedge\gamma\right)  =e^{\lambda}\left(  -\iota_{X}d\gamma
\wedge\gamma+d\lambda\wedge\gamma\right) \\
& =e^{\lambda}\left(  \left(  -\iota_{X}d\gamma+d\lambda\right)  \wedge
\gamma\right)  =e^{\lambda}\left(  \left(  \iota_{X}d\lambda\right)
\gamma\right)  \wedge\gamma=0.
\end{align*}

\end{proof}

\begin{example}
Let $M$ be a compact manifold and $H_{DR}^{k}\left(  M\right)  $ its de Rham
cohomology group of degree $k$. Suppose that there exists $\ \tau_{1}%
,\cdot\cdot\cdot,\tau_{p}$ closed $1$-forms on $M$ such that their classes
$\left[  \tau_{1}\right]  ,\cdot\cdot\cdot,\left[  \tau_{p}\right]  $ form a
basis of $H_{DR}^{1}\left(  M\right)  $ and such that $\left[  \tau_{j}%
\wedge\tau_{k}\right]  ,j,k=1,\cdot\cdot\cdot,p,\ j<k$, are linearly
independent in $H_{DR}^{2}\left(  M\right)  $. Let $L=S^{1}\times M$ endowed
with the product foliation given by $\xi=\ker ds$ where $\left(  s,X\right)  $
are variables in $S^{1}\times M$. The following assertions are equivalent:

i) $\beta\left(  s,X\right)  =a\left(  s\right)
{\displaystyle\sum\limits_{j=1}^{p}}
c_{j}\tau_{j}\left(  X\right)  $, $c_{j}\in\mathbb{R}$, $\left(  s,X\right)
\in L$.

ii) There exists a curve $\Gamma$ with values in $\mathfrak{MC}_{\delta
}\left(  L\right)  /\thicksim_{\mathcal{G}}$ such that the tangent to $\Gamma$
at the origin is $\left[  \beta\right]  $.

In particular $T_{\left[  0\right]  }\left(  \mathfrak{MC}_{\delta}\left(
L\right)  /\thicksim_{\mathcal{G}}\right)  =C^{\infty}\left(  S^{1}\right)
\times H_{DR}^{1}\left(  M\right)  /\mathbb{R}^{\ast}$ where the action of
$\mathbb{R}^{\ast}$is given by $\lambda\left(  a,h\right)  =\left(  \lambda
a,\lambda^{-1}h\right)  $.
\end{example}

\begin{proof}
We consider the DGLA defining couple $\left(  \gamma,X\right)  =\left(
ds,\frac{\partial}{\partial s}\right)  $.

$i)\implies ii)$. Let $\beta\left(  s,X\right)  =a\left(  s\right)
{\displaystyle\sum\limits_{j=1}^{p}}
c_{j}\tau_{j}\left(  X\right)  $. Take $\alpha_{t}=\beta t$. Then $\alpha
_{t}\in Z^{1}\left(  L\right)  $ and $\delta\beta=d_{b}\beta=d_{X}\beta=0$.

Moreover%
\[
\left\{  \beta,\beta\right\}  =2\iota_{X}d\beta\wedge\beta=2\iota
_{\frac{\partial}{\partial s}}\left(  a^{\prime}ds\wedge%
{\displaystyle\sum\limits_{j=1}^{p}}
c_{j}\tau_{j}+a%
{\displaystyle\sum\limits_{j=1}^{p}}
c_{j}d\tau_{j}\right)  \wedge a%
{\displaystyle\sum\limits_{j=1}^{p}}
c_{j}\tau_{j}=0.
\]
So $\alpha_{t}\in\mathfrak{MC}_{\delta}\left(  L\right)  $ and we can consider
$\Gamma:t\rightarrow\left[  \alpha_{t}\right]  \in\mathfrak{MC}_{\delta
}\left(  L\right)  /\thicksim_{\mathcal{G}}$.

$ii)\implies i)$. Let $\alpha_{t}=t\beta+t^{2}\sigma+o\left(  t^{2}\right)
\in\mathfrak{MC}_{\delta}\left(  L\right)  $. Then%
\[
\left\{  \alpha_{t},\alpha_{t}\right\}  =t^{2}\left\{  \beta,\beta\right\}
+o\left(  t^{2}\right)
\]
and
\[
\delta\alpha_{t}=d_{b}\alpha_{t}=td_{X}\beta+t^{2}d_{X}\sigma+o\left(
t^{2}\right)  .
\]
Since $\alpha_{t}\in\mathfrak{MC}_{\delta}\left(  L\right)  $, we obtain
$d_{X}\beta=0$ and $\left\{  \beta,\beta\right\}  +2d_{X}\sigma=0$, so
$\left[  \beta\right]  \in H_{DR}^{1}\left(  M\right)  $ and $\left[  \left\{
\beta,\beta\right\}  \right]  =0\in H_{DR}^{2}\left(  M\right)  $.

Since $\iota_{\frac{\partial}{\partial s}}\beta=0$ we have
\[
\beta\left(  s,X\right)  =%
{\displaystyle\sum\limits_{j=1}^{p}}
\beta_{j}\left(  s\right)  \tau_{j}+d_{X}f\left(  s,X\right)  ,\ f\in
C^{\infty}\left(  L\right)  .
\]
By Proposition \ref{Tangent Inclus Cohomologie} we may suppose $\beta\left(
s,X\right)  =%
{\displaystyle\sum\limits_{j=1}^{p}}
\beta_{j}\left(  s\right)  \tau_{j}\left(  X\right)  $. Then
\[
d\beta=%
{\displaystyle\sum\limits_{j=1}^{p}}
\beta_{j}^{\prime}ds\wedge\tau_{j}%
\]
and
\[
\left\{  \beta,\beta\right\}  =2\iota_{\frac{\partial}{\partial s}}%
d\beta\wedge\beta=2\left(
{\displaystyle\sum\limits_{j=1}^{p}}
\beta_{j}^{\prime}\tau_{j}\right)  \wedge\left(
{\displaystyle\sum\limits_{j=1}^{p}}
\beta_{j}\tau_{j}\right)  =2%
{\displaystyle\sum\limits_{j\neq k}}
\beta_{j}^{\prime}\beta_{k}\tau_{j}\wedge\tau_{k}%
\]
But
\[
\left[  \left\{  \beta,\beta\right\}  \right]  =2%
{\displaystyle\sum\limits_{j<k}}
\left(  \beta_{j}^{\prime}\beta_{k}-\beta_{k}^{\prime}\beta_{j}\right)
\left[  \tau_{j}\wedge\tau_{k}\right]  =0\in H_{DR}^{2}\left(  M\right)
\]
and from the assumption of linear independence it follows that $\beta
_{j}^{\prime}\beta_{k}-\beta_{k}^{\prime}\beta_{j}=0$ for every $1\leq j<k\leq
p$. This means that $\beta_{j}=c_{j}a$, $c_{j}\in\mathbb{R}$, $a\in C^{\infty
}\left(  S^{1}\right)  $ and $\beta\left(  s,X\right)  =a\left(  s\right)
{\displaystyle\sum\limits_{j=1}^{p}}
c_{j}\tau_{j}\left(  X\right)  $, $\left(  s,X\right)  \in L$.
\end{proof}

\begin{remark}
In the previous example we have $T_{\left[  0\right]  }\left(  \mathfrak{MC}%
_{\delta}\left(  L\right)  /\thicksim_{\mathcal{G}}\right)  \neq H^{1}\left(
\mathcal{Z}\left(  L\right)  ,\delta\right)  $ so the deformation theory is
obstructed at $\left[  0\right]  $. The hypothesis are fulfilled in the
particular case where $M$ is a torus.
\end{remark}

\section{Deformations of Levi-flat hypersurfaces}

\subsection{\label{paragraph}Maurer-Cartan equation for Levi-flat
deformations\bigskip}

\ \ \ \ \ \ \ \ \ \ \ \ \newline

Let $M$ be a complex manifold and $L$ a Levi flat hypersurface of class
$C^{\infty}$ in $M$ such that the Levi foliation of $M$ is co-orientable. In
this case there exists $r\in C^{\infty}\left(  M\right)  $, $dr\neq0$ on $L$
such that $L=\left\{  z\in M:\ r\left(  z\right)  =0\right\}  $ and set
$j:L\rightarrow M$ the natural inclusion. As $dr\neq0$ on a neighborhood of
$L$ in $M$ we will suppose in the sequel that $dr\neq0$ on $M$.

We denote by $J$ the complex structure on $M$. Then the distribution
$\xi=T\left(  L\right)  \cap JT\left(  L\right)  $ is integrable and $\xi
=\ker\gamma$, where $\gamma=j^{\ast}\left(  d_{J}^{c}r\right)  $. Since
$d_{J}^{c}=J^{-1}dJ$, we have $d_{J}^{c}r=-Jdr$.

Let $g$ be a fixed Hermitian metric on $M$ and $Z=grad_{g}r/\left\Vert
grad_{g}r\right\Vert _{g}^{2}$. Then the vector field $X=JZ$ is tangent to $L
$ and verifies
\[
\gamma\left(  X\right)  =d_{J}^{c}r\left(  JZ\right)  =1.
\]

It follows that the couple $\left(  \gamma,X\right)  $ defined above is a DGLA
defining couple for the Levi foliation. For a given defining function, we will
fix this DGLA defining couple and when its dependence on the defining function
$r$ has to be emphasised, we will say the DGLA defining couple associated to
$r$.

Let $U$ be a tubular neighborhood of $L$ in $M$ and $\pi:U\rightarrow L$ the
projection on $L$ along the integral curves of $Z$. As we are interested in
infinitesimal deformations we may suppose $U=M$.

We will now parametrize the real hypersurfaces near $L$ and diffeomorphic to
$L$ as graphs over $L$:

Let $\mathcal{F=}C^{\infty}\left(  L;\mathbb{R}\right)  $ and $a\in
\mathcal{F}$. Denote
\[
L_{a}=\left\{  z\in M:\ r\left(  z\right)  =a\left(  \pi\left(  z\right)
\right)  \right\}  .
\]

Since $Z$ is transverse to $L$, $L_{a}$ is a hypersurface in $M$. Consider the
map $\Phi_{a}:M\rightarrow M$ defined by $\Phi_{a}\left(  p\right)  =q$,
where
\begin{equation}
\pi\left(  q\right)  =\pi\left(  p\right)  ,\ r\left(  q\right)  =r\left(
p\right)  +a\left(  \pi\left(  p\right)  \right)  .\label{Def Fi a}%
\end{equation}
$U$ is a tubular neighborhood of $L$, so $\Phi_{a}$ is a diffeomorphism of $M$
such that $\Phi_{a}\left(  L\right)  =L_{a}$ and $\Phi_{a}^{-1}=\pi\left\vert
_{L_{a}}\right.  $.

Conversely, let $\Psi\in\mathcal{U}\subset\mathcal{G=}Diff\left(  M\right)  $,
where $\mathcal{U}$ is a suitable neighborhood of the identity in
$\mathcal{G}$ as in Definition \ref{Def group action}. Then there exists
$a\in\mathcal{F}$ such that $\Psi\left(  L\right)  =L_{a}$. Indeed, for $X\in
L$, let $q\left(  X\right)  \in\Psi\left(  L\right)  $ such that $\pi\left(
q\left(  X\right)  \right)  =X$. By defining $a\left(  X\right)  =r\left(
q\left(  X\right)  \right)  $, we obtain $\Psi\left(  L\right)  =L_{a}$.

So we have the following:

\begin{lemma}
Let $\ \Psi\in\mathcal{U}$. Then there exists a unique $a\in\mathcal{F}$ such
that $\Psi\left(  L\right)  =L_{a}$.
\end{lemma}

It follows that a neighborhood $\mathcal{V}_{\mathcal{F}}$ of $0$ in
$\mathcal{F}$ is a set of parametrization of hypersurfaces close to $L.$

For $a\in\mathcal{V}_{\mathcal{F}}$, consider the almost complex structure
$J_{a}=\left(  \Phi_{a}^{-1}\right)  _{\ast}\circ J\circ\left(  \Phi
_{a}\right)  _{\ast}$ on $M$ and denote
\begin{equation}
\alpha_{a}=\left(  d_{J_{a}}^{c}r\left(  X\right)  \right)  ^{-1}j^{\ast
}\left(  d_{J_{a}}^{c}r\right)  -\gamma.\label{def alfa a}%
\end{equation}
Then $\alpha_{a}\in\mathcal{Z}^{1}\left(  L\right)  $ and
\begin{equation}
\ker\left(  \gamma+\alpha_{a}\right)  =\ker j^{\ast}\left(  d_{J_{a}}%
^{c}r\right)  =TL\cap J_{a}TL.\label{05}%
\end{equation}
Let $V\in TL\cap J_{a}TL$. Then $V=Y+\theta X$ with $Y\in TL\cap JTL$ and
$\theta$ a real function on $L$. By (\ref{05}) we have%
\[
d_{J_{a}}^{c}r\left(  V\right)  =j^{\ast}d_{J_{a}}^{c}r\left(  Y\right)
+\theta j^{\ast}d_{J_{a}}^{c}r\left(  X\right)  =0,
\]
so%
\[
\theta=-\left(  d_{J_{a}}^{c}r\left(  X\right)  \right)  ^{-1}d_{J_{a}}%
^{c}r\left(  Y\right)  =-\alpha_{a}\left(  Y\right)
\]
and it follows that%
\begin{equation}
TL\cap J_{a}TL=\left\{  Y-\left(  \alpha_{a}\left(  Y\right)  \right)
X:\ Y\in TL\cap JTL\right\}  .\label{06}%
\end{equation}
Since%
\begin{equation}
\pi_{\ast}\left(  TL_{a}\cap JTL_{a}\right)  =\left(  \Phi_{a}^{-1}\right)
_{\ast}\left(  TL_{a}\cap JTL_{a}\right)  =TL\cap\left(  \Phi_{a}^{-1}\right)
_{\ast}\left(  J\left(  \Phi_{a}\right)  _{\ast}TL\right)  =TL\cap
J_{a}TL\label{07}%
\end{equation}
from (\ref{05}), (\ref{06}) and (\ref{07}) we obtain the following

\begin{lemma}
\label{ker gama+alfa(a)}For every $a\in\mathcal{V}_{\mathcal{F}}$ the form
$\alpha_{a}$ is the unique form in $\mathcal{Z}^{1}\left(  L\right)  $
verifying
\[
\ker\left(  \gamma+\alpha_{a}\right)  =\pi_{\ast}\left(  TL_{a}\cap
JTL_{a}\right)  .
\]
Moreover,
\begin{align*}
\ker\left(  \gamma+\alpha_{a}\right)   & =\ker j^{\ast}\left(  d_{J_{a}}%
^{c}r\right)  =\pi_{\ast}\left(  TL_{a}\cap JTL_{a}\right)  =TL\cap J_{a}TL\\
& =\left\{  Y-\left(  \alpha_{a}\left(  Y\right)  \right)  JZ:\ Y\in TL\cap
JTL\right\}  .
\end{align*}

\end{lemma}

By using Lemma \ref{ker gama+alfa(a)} and Lemma
\ref{Ker(gama+alfa) integrable} we can state the following

\begin{corollary}
\label{MC Levi flat cor}For every $a\in\mathcal{V}_{\mathcal{F}}$, the
following assertions are equivalent:

i) $L_{a}$ is Levi flat.

ii) $\alpha_{a}$ satisfies \ the Maurer Cartan equation in $\left(
\mathcal{Z}^{\ast}\left(  L\right)  ,\delta,\left\{  \cdot,\cdot\right\}
\right)  $ i.e.%
\begin{equation}
\delta\alpha_{a}+\frac{1}{2}\left\{  \alpha_{a},\alpha_{a}\right\}
=0.\label{MC Levi flat}%
\end{equation}

\end{corollary}

\begin{remark}
Suppose now that $a,b\in\mathcal{V}_{\mathcal{F}}$, $\Phi\in\mathcal{G}%
=Diff\left(  L\right)  $ and $\chi\left(  \Phi\right)  \left(  \alpha
_{a}\right)  =\alpha_{b}$, where $\chi\left(  \Phi\right)  $ is the group
action defined in (\ref{chi(fi)}). From Lemma \ref{ker gama+alfa(a)} and
Remark \ref{group action} it follows that $L_{a}$\ is Levi flat if and only if
$L_{b}$\ is Levi flat.
\end{remark}

\begin{notation}
Set $\mathcal{E}\mathfrak{=}\left\{  \alpha_{a}:\ a\in\mathcal{V}%
_{\mathcal{F}}\right\}  $.
\end{notation}

\begin{remark}
$\mathcal{E}$ parametrizes the codimension $1$ distributions close to $TL\cap
JTL$ which are of the form $TL\cap\widetilde{J}TL$ for $\widetilde{J} $
complex structure (possible non integrable) close to $J$, where $\widetilde
{J}=\left(  I+S\right)  J\left(  I+S\right)  ^{-1}$ with $S\in\Lambda
_{J}^{0,1}\left(  M\right)  \otimes T\left(  M\right)  $ close to $0$.
\end{remark}

By using the notations of Definition \ref{def moduli foliations} , we are now
able to put in evidence the moduli space of deformations of Levi-flat
manifolds of $L$:

\begin{definition}
Let $\mathcal{R}_{\mathcal{F}}=\left\{  \zeta\in\mathcal{D}:\ \zeta
=\ker\left(  \gamma+\beta\right)  ,\ \beta\in\mathcal{E}\right\}  $. The
moduli space of deformations of Levi-flat manifolds of $L$ is $\pi^{-1}\left(
\pi\left(  \mathcal{I\cap R}_{\mathcal{F}}\right)  \right)  /\mathcal{G}$.
\end{definition}

\begin{remark}
From Remark \ref{local action} it follows that the local corresponding action
of $\mathcal{G}$ on $\mathcal{E}$ is given by $\alpha_{b}=\chi\left(
\Phi\right)  \left(  \alpha_{a}\right)  $, where $a,b\in\mathcal{V}%
_{\mathcal{F}}$ and $\Phi\in\mathcal{G}$ is sufficiently close to the
identity. If $r$, $r^{\prime}\in C^{\infty}\left(  M\right)  $, $dr\neq0$,
$dr^{\prime}\neq0$ on $L$ such that $L=\left\{  z\in M:\ r\left(  z\right)
=0\right\}  =\left\{  z\in M:\ r^{\prime}\left(  z\right)  =0\right\}  $,
$r=hr^{\prime}$ with $h>0$ of class $C^{\infty}$ in a neighborhood $L$. So
$\left\{  z\in M:\ r\left(  z\right)  =a\left(  \pi\left(  z\right)  \right)
\right\}  =\left\{  z\in M:\ r^{\prime}\left(  z\right)  =h^{-1}\left(
z\right)  a\left(  \pi\left(  z\right)  \right)  \right\}  $.\ It follows that
the previous definition does not depend on the choice of the defining function
$r $ of $L$ and by Proposition \ref{Iso DGVS} it follows that it does not
depent nor on the choice of the metric $g$. We remark also that the moduli
space of deformations of Levi-flat manifolds of $L$ identifies Levi flat
hypersurfaces up to a foliated diffeomorphism and not up to a CR diffeomorphism.
\end{remark}

\ \ \ \ \ \ \ \ \ \ \ \newline

\subsection{Equations for infinitesimal Levi-flat deformations\bigskip}

\ \ \ \ \ \ \ \ \ \ \ \ \newline

Let $M$ be a complex manifold, $J$ the complex structure on $M$, $L$ a Levi
flat hypersurface in $M$ and $I$ an open interval in $\mathbb{R}$ containing
the origin. A $1$-dimensional Levi-flat deformation of $L$ is a smooth mapping
$\Psi:I\times M\rightarrow M$ such that $\Psi_{t}=\Psi\left(  t,\cdot\right)
\in Diff\left(  M\right)  $,\ $L_{t}=\Psi_{t}L$ is a Levi flat hypersurface in
$M$ for every $t\in I$ \ and $L_{0}=L$. By the previous subsection there
exists a family $\left(  a_{t}\right)  _{t\in I}$ in $\mathcal{V}%
_{\mathcal{F}}$ such that $\pi_{\ast}\left(  TL_{a_{t}}\cap JTL_{a_{t}%
}\right)  =\ker\left(  \gamma+\alpha_{a_{t}}\right)  $ and $\alpha_{a_{t}}$
satisfies the Maurer Cartan equation (\ref{MC Levi flat}) in $\left(
\mathcal{Z}^{\ast}\left(  L\right)  ,\delta,\left\{  \cdot,\cdot\right\}
\right)  $ for every $t$. We will say that the family $\left(  a_{t}\right)
_{t\in I}$ is a family in $\mathcal{V}_{\mathcal{F}}$ defining a Levi-flat
deformation of $L$.

We define now $\delta^{c}:\mathcal{Z}^{\ast}\left(  L\right)  \rightarrow
\mathcal{Z}^{\ast}\left(  L\right)  $: for $\alpha\in\mathcal{Z}^{p}\left(
L\right)  $ and $V_{1},\cdot\cdot\cdot,V_{p+1}\in T\left(  L\right)  \cap
JT\left(  L\right)  $ set $\delta^{c}\alpha\left(  V_{1},\cdot\cdot
\cdot,V_{p+1}\right)  =J^{-1}\delta J\alpha\left(  V_{1},\cdot\cdot
\cdot,V_{p+1}\right)  $ and $\delta^{c}\alpha\left(  X,V_{1},\cdot\cdot
\cdot,V_{p}\right)  =0$. By extending this definition by linearity we obtain
$\delta^{c}\alpha\in\mathcal{Z}^{p+1}\left(  L\right)  $.

Recall that $\left(  \gamma,X\right)  $ is a DGLA defining couple, where
$\gamma=j^{\ast}\left(  d_{J}^{c}r\right)  $ and $X=JZ=J\left(  grad_{g}%
r/\left\Vert grad_{g}r\right\Vert ^{2}\right)  $, $r$ is a defining function
for $L$ and $g$ a Hermitian metric on $M$.

\begin{proposition}
\label{d/dt(alfa(a))=delta c p}Let $L$ be a Levi flat hypersurface in a
complex manifold $M$, $\left(  a_{t}\right)  _{t\in I}$ a family in
$\mathcal{V}_{\mathcal{F}}$ defining a Levi-flat deformation of $L$ and
$p=\frac{da_{t}}{dt}_{\left\vert t=0\right.  }$.Then
\[
\frac{d\alpha_{a_{t}}}{dt}_{\left\vert t=0\right.  }=\delta^{c}p.\
\]

\end{proposition}

\begin{proof}
Since $\alpha_{a_{t}}\left(  X\right)  =0$ for every $t$ it follows that
\begin{equation}
\frac{d\alpha_{a_{t}}}{dt}_{\left\vert t=0\right.  }\left(  X\right)
=0=\left(  \delta^{c}p\right)  \left(  X\right)  .\label{d/dt alfa JZ}%
\end{equation}
Let $V$ be a section of $TL\cap JTL$, which will be identified for simplicity
with $j_{\ast}V$. Then (\ref{def alfa a}) gives
\begin{align*}
\frac{d\alpha_{a_{t}}}{dt}_{\left\vert t=0\right.  }\left(  V\right)   &
=\frac{d}{dt}_{\left\vert t=0\right.  }\left(  \left(  d_{J_{a_{t}}}%
^{c}r\left(  X\right)  \right)  ^{-1}\right)  j^{\ast}\left(  d_{J_{a_{0}}%
}^{c}r\right)  \left(  V\right) \\
& +\left(  d_{J_{a_{0}}}^{c}\left(  JZ\right)  \right)  ^{-1}\frac{d}%
{dt}_{\left\vert t=0\right.  }j^{\ast}\left(  d_{J_{a_{t}}}^{c}r\right)
\left(  V\right)  .
\end{align*}
But%
\[
j^{\ast}\left(  d_{J_{a_{0}}}^{c}r\right)  \left(  V\right)  =j^{\ast}\left(
d_{J}^{c}r\right)  \left(  V\right)  =0
\]
and%
\[
\left(  d_{J_{a_{0}}}^{c}r\left(  X\right)  \right)  ^{-1}=\left(  d_{J}%
^{c}r\left(  X\right)  \right)  ^{-1}=1,
\]
so%
\begin{align}
\frac{d\alpha_{a_{t}}}{dt}_{\left\vert t=0\right.  }\left(  V\right)   &
=\frac{d}{dt}_{\left\vert t=0\right.  }j^{\ast}\left(  d_{J_{a_{t}}}%
^{c}r\right)  \left(  V\right)  =\frac{d}{dt}_{\left\vert t=0\right.  }\left(
-J_{a_{t}}dr\right)  \left(  V\right) \nonumber\\
& =-\left(  dr\right)  \frac{d}{dt}_{\left\vert t=0\right.  }\left(  J_{a_{t}%
}V\right)  .\label{d/dt afa (t)}%
\end{align}
We have%
\begin{align}
\frac{d}{dt}_{\left\vert t=0\right.  }\left(  J_{a_{t}}V\right)   & =\frac
{d}{dt}_{\left\vert t=0\right.  }\left(  \left(  \Phi_{a_{t}}^{-1}\right)
_{\ast}\circ J\circ\left(  \Phi_{a_{t}}\right)  _{\ast}\right)  \left(
V\right) \nonumber\\
& =\frac{d}{dt}_{\left\vert t=0\right.  }\left(  \Phi_{a_{t}}^{-1}\right)
_{\ast}\left(  JV\right)  +J\frac{d}{dt}_{\left\vert t=0\right.  }\left(
\Phi_{a_{t}}\right)  _{\ast}\left(  V\right)  .\label{d/dt Ja(t)}%
\end{align}
By using the definition (\ref{Def Fi a}) of $\Phi_{a_{t}}$ we have%
\[
r\left(  \Phi_{a_{t}}\left(  z\right)  \right)  =r\left(  z\right)
+a_{t}\left(  \pi\left(  z\right)  \right)  =r\left(  z\right)  +tp\left(
\pi\left(  z\right)  \right)  +o\left(  t\right)  ,
\]
where $\pi$ is the projection along the integral curves of $Z$. It follows
that
\begin{equation}
\frac{d\left(  \Phi_{a_{t}}\right)  _{\ast}}{dt}_{\left\vert t=0\right.
}=\left(  p\circ\pi\right)  Z.\label{p o pi}%
\end{equation}
If we consider a smoth extension $\widetilde{p}$ of $p$ to $M$ and the flow
$\Phi^{\widetilde{p}Z}$ of $\widetilde{p}Z$, we have
\[
\frac{d\Phi_{t}^{\widetilde{p}Z}}{dt}\left(  z\right)  =\left(  \widetilde
{p}Z\right)  \left(  \Phi_{t}^{\widetilde{p}Z}\left(  z\right)  \right)
\]
and restricting to $L$, by (\ref{p o pi}) we obtain%
\begin{equation}
\frac{d\left(  \Phi_{a_{t}}\right)  _{\ast}}{dt}_{\left\vert t=0\right.
}=\frac{d\left(  \Phi_{t}^{\widetilde{p}Z}\right)  _{\ast}}{dt}_{\left\vert
t=0\right.  }=pZ.\label{deriv=flow}%
\end{equation}
So (\ref{d/dt Ja(t)}) and (\ref{deriv=flow}) give%
\begin{align*}
\frac{d}{dt}_{\left\vert t=0\right.  }\left(  J_{a_{t}}V\right)   & =\frac
{d}{dt}_{\left\vert t=0\right.  }\left(  \Phi_{-t}^{pZ}\right)  _{\ast}\left(
JV\right)  +J\frac{d}{dt}_{\left\vert t=0\right.  }\left(  \Phi_{t}%
^{pZ}\right)  _{\ast}\left(  V\right) \\
& =-\mathcal{L}_{pZ}\left(  JV\right)  +J\mathcal{L}_{pZ}\left(  V\right) \\
& =-\left[  pZ,JV\right]  +J\left[  pZ,V\right] \\
& =-p\left[  Z,JV\right]  +JV\left(  p\right)  Z+pJ\left[  Z,V\right]
-V\left(  p\right)  JZ.
\end{align*}
Replacing this formula in (\ref{d/dt afa (t)}) we obtain
\[
\frac{d}{dt}_{\left\vert t=0\right.  }\alpha_{a_{t}}\left(  V\right)
=-\left(  dr\right)  \left(  -p\left[  Z,JV\right]  +JV\left(  p\right)
Z+pJ\left[  Z,V\right]  -V\left(  p\right)  JZ\right)  .
\]
Since $dr\left(  JZ\right)  =0$ and $dr\left(  Z\right)  =1$ it follows that%
\begin{align}
\frac{d}{dt}_{\left\vert t=0\right.  }\alpha_{a_{t}}\left(  V\right)   &
=\left(  dr\right)  \left(  p\left[  Z,JV\right]  \right)  -JV\left(
p\right)  -p\left(  dr\right)  J\left[  Z,V\right] \nonumber\\
& =pdr\left(  \left[  Z,JV\right]  \right)  -JV\left(  p\right)  +p\left(
d^{c}r\right)  \left[  Z,V\right]  .\label{d/dt afa (t) 0}%
\end{align}
By using
\[
0=ddr\left(  Z,JV\right)  =Z\left(  dr\left(  JV\right)  \right)  -JV\left(
dr\left(  Z\right)  \right)  -dr\left[  Z,JV\right]
\]
we obtain%
\[
dr\left[  Z,JV\right]  =0
\]
and (\ref{d/dt afa (t) 0}) becomes
\begin{equation}
\frac{d}{dt}_{\left\vert t=0\right.  }\alpha_{a_{t}}\left(  V\right)
=-JV\left(  p\right)  +p\left(  d^{c}r\right)  \left[  Z,V\right]
.\label{d/dt afa (t) 1}%
\end{equation}
Since $d^{c}r\left(  V\right)  =-dr\left(  JV\right)  =0$ and $d^{c}r\left(
Z\right)  =-dr\left(  JZ\right)  =0$, it follows that$\ $%
\[
dd^{c}r\left(  Z,V\right)  =Z\left(  d^{c}r\left(  V\right)  \right)
-V\left(  d^{c}r\left(  Z\right)  \right)  -d^{c}r\left(  \left[  Z,V\right]
\right)  =-d^{c}r\left(  \left[  Z,V\right]  \right)
\]
and from (\ref{d/dt afa (t) 1}) we deduce%
\begin{equation}
\frac{d}{dt}_{\left\vert t=0\right.  }\alpha_{a_{t}}\left(  V\right)
=-JV\left(  p\right)  -pdd^{c}r\left(  Z,V\right)  =d^{c}p\left(  V\right)
-pJ\left(  \iota_{JZ}dd^{c}r\right)  \left(  V\right)  .\label{d/dt afa (t) 3}%
\end{equation}
Now%
\begin{equation}
\left(  \delta^{c}p\right)  \left(  V\right)  =-\delta p\left(  JV\right)
=d^{c}p\left(  V\right)  -\left\{  \gamma,p\right\}  \left(  JV\right)
\label{delta p c}%
\end{equation}
and%
\begin{equation}
\left\{  \gamma,p\right\}  \left(  JV\right)  =p\mathcal{L}_{X}\gamma\left(
JV\right)  -\left(  \mathcal{L}_{X}p\right)  \gamma\left(  JV\right)
.\label{(gama,p)}%
\end{equation}
Since $\gamma\left(  JV\right)  =0$ and $\iota_{X}\gamma=1$, (\ref{(gama,p)})
becomes
\begin{equation}
\left\{  \gamma,p\right\}  \left(  JV\right)  =p\iota_{X}d\gamma\left(
JV\right)  .\label{(gama,p)1}%
\end{equation}
Therefore, recalling now that $\gamma=j^{\ast}\left(  d^{c}r\right)  $ and
$X=JZ$, from (\ref{(gama,p)1}) we obtain%
\[
\left\{  \gamma,p\right\}  \left(  JV\right)  =p\left(  \iota_{JZ}%
dd^{c}r\right)  \left(  JV\right)
\]
and from (\ref{delta p c}) it follows that
\begin{equation}
\left(  \delta^{c}p\right)  \left(  V\right)  =d^{c}p\left(  V\right)
-p\left(  \iota_{JZ}dd^{c}r\right)  \left(  JV\right)  .\label{delta p c 1}%
\end{equation}
Finally, by (\ref{d/dt afa (t) 3}), (\ref{delta p c 1}) and
(\ref{d/dt alfa JZ}) we conclude
\begin{equation}
\frac{d\alpha_{a_{t}}}{dt}_{\left\vert t=0\right.  }=\delta^{c}%
p\ .\label{d/dt afa (t) 4}%
\end{equation}

\end{proof}

\begin{notation}
For a DGLA defining couple $\left(  \gamma,X\right)  $ we denote
$\mathfrak{b}=\iota_{X}d\gamma$. By Lemma \ref{i(X) d(gama ) db closed},
$\mathfrak{b}$ is $d_{b}$-closed and $c\left(  T\left(  L\right)  \cap
JT\left(  L\right)  \right)  =\left[  \mathfrak{b}\right]  \in H^{1}\left(
\Lambda^{\ast}\left(  \xi\right)  ,d_{b}\right)  $. Let $F$ be a compact leaf
of the Levi foliation. Then there exists a unique harmonic form $\mathfrak{b}%
_{F}\in\Lambda^{1}\left(  F\right)  $ with respect to the fixed metric $g$
such that $\left[  \mathfrak{b}_{\left\vert F\right.  }\right]  =\left[
\mathfrak{b}_{F}\right]  \in H^{1}\left(  F,d_{b}\right)  $, where
$\mathfrak{b}_{\left\vert F\right.  }$ is the restriction of $\mathfrak{b}$ to
$F$.
\end{notation}

\begin{corollary}
\label{equation p general}Let $L$ be a Levi flat hypersurface in a complex
manifold $M$, $\left(  a_{t}\right)  _{t\in I}$ a family in $\mathcal{V}%
_{\mathcal{F}}$ defining a Levi-flat deformation of $L$ and $p=\frac{da_{t}%
}{dt}_{\left\vert t=0\right.  }$. Then:%
\begin{equation}
\delta\delta^{c}p=0.\label{Eq p 1}%
\end{equation}

or equivalently%
\begin{equation}
d_{b}d_{b}^{c}p-d_{b}p\wedge J\mathfrak{b}-d_{b}^{c}p\wedge\mathfrak{b}%
-pJd_{b}^{c}\mathfrak{b}-p\mathfrak{b}\wedge J\mathfrak{b}=0.\label{Eq p A}%
\end{equation}

\end{corollary}

\begin{proof}
$\alpha_{a_{t}}$verifies the Maurer Cartan equation (\ref{MC Levi flat}) in
$\left(  \mathcal{Z}^{\ast}\left(  L\right)  ,\delta,\left\{  \cdot
,\cdot\right\}  \right)  $ so%
\[
\delta\alpha_{a_{t}}+\frac{1}{2}\left\{  \alpha_{a_{t}},\alpha_{a_{t}%
}\right\}  =0
\]
for every $t$.\ Since
\[
\frac{d}{dt}_{\left\vert t=0\right.  }\left\{  \alpha_{a_{t}},\alpha_{a_{t}%
}\right\}  =0,
\]
(\ref{Eq p 1}) follows from (\ref{d/dt afa (t) 4}).

By (\ref{rel delta db}) we have
\[
\delta^{c}p=-J\delta p=-J\left(  d_{b}p+p\iota_{X}d\gamma\right)  =d_{b}%
^{c}p-pJ\mathfrak{b}%
\]
and%
\begin{align*}
\delta\delta^{c}p  & =\delta\left(  d_{b}^{c}p-pJ\mathfrak{b}\right)
=d_{b}\left(  d_{b}^{c}p-pJ\mathfrak{b}\right)  +\mathfrak{b}\wedge\left(
d_{b}^{c}p-pJ\mathfrak{b}\right) \\
& =d_{b}d_{b}^{c}p-d_{b}p\wedge J\mathfrak{b}-pd_{b}J\mathfrak{b}-d_{b}%
^{c}p\wedge\mathfrak{b}-p\mathfrak{b}\wedge J\mathfrak{b}.
\end{align*}
So (\ref{Eq p 1}) is equivalent to (\ref{Eq p A}).
\end{proof}

\begin{proposition}
\label{beta harmonique}Let $M$ be a complex manifold and $L$ a $C^{\infty}$
Levi flat hypersurface in $M$. Let $F$ be a compact leaf of the Levi
foliation. Then there exists a defining function $\rho$ of $L$ such that the
DGLA defining couple $\left(  \widehat{\gamma},\widehat{X}\right)  $
associated to $\rho$ verifies%
\begin{equation}
\mathfrak{b}_{F}=\iota_{\widehat{X}}d\widehat{\gamma}_{\left\vert F\right.
}=\iota_{\widehat{X}}\left(  d_{b}d_{b}^{c}\rho\right)  _{\left\vert F\right.
}.\label{beta F 0}%
\end{equation}

\end{proposition}

\begin{proof}
Let $r$ be a $C^{\infty}$ defining function for $L$ and $\left(
\gamma,X\right)  $ the DGLA defining couple associated to $r$.

Since $\left[  \mathfrak{b}_{\left\vert F\right.  }\right]  =\left[
\mathfrak{b}_{F}\right]  \in H^{1}\left(  F,d_{b}\right)  $, there exists
$\lambda\in C^{\infty}\left(  F\right)  $ such that
\[
\mathfrak{b}_{F}=\mathfrak{b}_{\left\vert F\right.  }+d_{b}\lambda.
\]
By using (\ref{db}) we obtain%
\begin{equation}
\mathfrak{b}_{F}=\iota_{X}d_{b}d_{b}^{c}r_{\left\vert F\right.  }%
+d\lambda-\left(  \iota_{X}d\lambda\right)  j^{\ast}\left(  d^{c}r\right)
_{\left\vert F\right.  }.\label{beta F}%
\end{equation}
We chose a smooth extension of $\lambda$ on $M$ which we denote by $\lambda$
too, and set $\rho=e^{-\lambda}r$.

We have%
\[
d^{c}\left(  e^{-\lambda}r\right)  =e^{-\lambda}\left(  d^{c}r-rd^{c}%
\lambda\right)
\]
and%
\begin{equation}
dd^{c}\left(  e^{-\lambda}r\right)  =e^{-\lambda}\left(  -d\lambda\wedge
d^{c}r+rd\lambda\wedge d^{c}\lambda+dd^{c}r-dr\wedge d^{c}\lambda
-rdd^{c}\lambda\right)  .\label{ddc(e-lambdar)}%
\end{equation}

Let $V$ be a section of $TL\cap JTL$. Since $r=0$ on $L$, $j^{\ast}%
d^{c}r\left(  X\right)  =1$ and $j^{\ast}d^{c}r\left(  V\right)  =0$, from
(\ref{ddc(e-lambdar)}) we obtain
\begin{align}
\iota_{e^{\lambda}X}dd^{c}\left(  e^{-\lambda}r\right)  \left(  V\right)   &
=dd^{c}\left(  e^{-\lambda}r\right)  \left(  e^{\lambda}X,V\right) \nonumber\\
& =e^{-\lambda}\left(  \left(  -d\lambda\wedge d^{c}r\right)  \left(
e^{\lambda}X,V\right)  \right)  +dd^{c}r\left(  e^{\lambda}X,V\right)
\nonumber\\
& -dr\wedge d^{c}\lambda\left(  \left(  e^{\lambda}X,V\right)  \right)
\nonumber\\
& =\left(  \left(  -d\lambda\wedge d^{c}r\right)  \left(  X,V\right)
+\iota_{e^{\lambda}X}dd^{c}r\left(  V\right)  \right) \nonumber\\
& =\left(  d\lambda\left(  V\right)  +\iota_{e^{\lambda}X}dd^{c}r\left(
V\right)  \right)  .\label{beta F 1}%
\end{align}

But (\ref{beta F}) and (\ref{beta F 1}) give
\[
\iota_{e^{\lambda}X}dd^{c}\left(  e^{-\lambda}r\right)  \left(  V\right)
=\mathfrak{b}_{F}\left(  V\right)  \ on\ F
\]
and this equality proves (\ref{beta F 0}).
\end{proof}

\begin{proposition}
\label{Eq p Kahler}Let $L$ be a Levi flat hypersurface in a K\"{a}hler
manifold $M$, $\left(  a_{t}\right)  _{t\in I}$ a family in $\mathcal{V}%
_{\mathcal{F}}$ defining a Levi-flat deformation of $L$ and $p=\frac{da_{t}%
}{dt}_{\left\vert t=0\right.  }$. Let $F$ be a compact leaf of the Levi
foliation and $\partial_{b}$, $\overline{\partial}_{b}$ the tangential
operators along the leaves. Then%

\begin{equation}
d_{b}d_{b}^{c}p-d_{b}p\wedge J\mathfrak{b}_{F}-d_{b}^{c}p\wedge\mathfrak{b}%
_{F}-p\mathfrak{b}_{F}\wedge J\mathfrak{b}_{F}=0\label{Eq p K reelle}%
\end{equation}
or equivalently
\begin{equation}
\partial_{b}\overline{\partial}_{b}p+\partial_{b}p\wedge\overline{\theta}%
_{F}-\overline{\partial}_{b}p\wedge\theta_{F}-p\overline{\theta}_{F}%
\wedge\theta_{F}=0\ \label{Eq p K}%
\end{equation}
where
\[
\theta_{F}=\mathfrak{b}_{F}^{1,0}=\frac{1}{2}\left(  \mathfrak{b}%
_{F}-iJ\mathfrak{b}_{F}\right)  .
\]

\end{proposition}

\begin{proof}
We choose a defining function of $L$ as in \ Proposition \ref{beta harmonique}%
. We consider on $F$ the metric induced by the K\"{a}hler metric of $M$. Since
$\mathfrak{b}_{F}$ is a harmonic form on $F$ with respect to this K\"{a}hler
metric, it follows that $J\mathfrak{b}_{F}$ is also a harmonic form. So
$d_{b}J\mathfrak{b}_{F}=d_{b}^{c}J\mathfrak{b}_{F}=0$ and (\ref{Eq p K reelle}%
), (\ref{Eq p K}) follow from (\ref{Eq p A}).
\end{proof}

\ \ \ \ \ \ \ \ \ \newline

\subsection{A uniqueness theorem for partial differential equations\bigskip}

\ \ \ \ \ \ \ \ \ \ \ \ \ \ \ \newline

In this section we prove a uniqueness theorem for second order partial
differential equations on compact K\"{a}hler manifolds which will be used in
the next sections to give infinitesimal rigidity results for Levi flat hypersurfaces.

For $\ \varphi,\psi\in\Lambda^{k}\left(  M\right)  $, we use the notations%
\[
\ \left\langle \varphi,\psi\right\rangle =\varphi\wedge\ast\overline{\psi
},\ \left\langle \left\langle \varphi,\psi\right\rangle \right\rangle
=\int_{M}\left\langle \varphi,\psi\right\rangle ,\left\Vert \varphi\right\Vert
^{2}=\left\langle \left\langle \varphi,\varphi\right\rangle \right\rangle
,\ \left\Vert \varphi\right\Vert _{\infty}^{2}=\underset{M}{\sup}%
\ast\left\langle \varphi,\varphi\right\rangle ,
\]
where $\ast$ is the Hodge operator. If $T\in End\left(  \Lambda^{\ast
}M\right)  $, we denote $T^{c}=J^{-1}TJ$, where $J$ is the complex structure
of $M$.

\begin{theorem}
\label{uniqueness}Let $M$ be a compact K\"{a}hler manifold and $\beta\neq0$ a
harmonic $1$-form on $M$.\ Let $A\in End\left(  \Lambda^{\ast}M\right)
$\ defined by $A\alpha=\beta\wedge\alpha$ and $P=d+A$ . Suppose that
$\Delta-A^{\ast}A$ is positive defined on a subspace $E\subset\Lambda^{0}M $,
where $\Delta$ is the Laplace operator on $M$. Then $f=0$ is the unique
solution of the equation $PP^{c}f=0$, $f\in E$. In particular $\Delta-A^{\ast
}A$ is positive defined if $\left\Vert \beta\right\Vert _{\infty}^{2}%
<\lambda_{\Delta}^{1}$, where $\lambda_{\Delta}^{1}$ is the smallest strictly
positive eigenvalue of the Dirichlet form $f\mapsto\left\langle \left\langle
\bigtriangleup f,f\right\rangle \right\rangle $ and the conclusion of the
theorem is valid in this case.
\end{theorem}

\begin{proof}
Let $f\in E$ such that
\begin{equation}
PP^{c}f=dP^{c}f+\beta\wedge P^{c}f=0.\label{PPcf}%
\end{equation}
Let $\omega$ be the K\"{a}hler form on $M$ and $\Lambda:\Lambda^{k+2}%
M\rightarrow\Lambda^{k}M$ the adjoint of the exterior multiplication by
$\omega$, $\Lambda\alpha=\ast^{-1}\left(  \omega\wedge\ast\overline{\alpha
}\right)  $. Then (\ref{PPcf}) gives%
\begin{equation}
\Lambda dP^{c}f=-\Lambda\left(  \beta\wedge P^{c}f\right)  =-\left\langle
\omega,\beta\wedge P^{c}f\right\rangle .\label{Lambda dPcf}%
\end{equation}

Step1.
\begin{equation}
\left\langle \omega,\beta\wedge P^{c}f\right\rangle =\left\langle J\beta
,P^{c}f\right\rangle .\label{Prod sc}%
\end{equation}
Indeed, let $\left(  \theta_{1},\cdot\cdot\cdot,\theta_{n},J\theta_{1}%
,\cdot\cdot\cdot,J\theta_{n}\right)  $ a local orthonormal basis at $z $ for
$\Lambda^{1}\left(  M\right)  $ such that $\omega\left(  z\right)  =%
{\displaystyle\sum\limits_{j}}
d\theta_{j}\wedge dJ\theta$. Then by writing $\beta=%
{\displaystyle\sum\limits_{j}}
a_{j}d\theta_{j}+%
{\displaystyle\sum\limits_{j}}
b_{j}dJ\theta_{j}$, $P^{c}f=%
{\displaystyle\sum\limits_{j}}
c_{j}d\theta_{j}+%
{\displaystyle\sum\limits_{j}}
d_{j}dJ\theta_{j}$ , we have
\[
\left\langle \omega,\beta\wedge P^{c}f\right\rangle \left(  z\right)  =%
{\displaystyle\sum\limits_{j}}
\left(  a_{j}d_{j}-b_{j}c_{j}\right)  \left(  z\right)  dV=\left\langle
J\beta,P^{c}f\right\rangle \left(  z\right)  .
\]

Step 2. Let $B=d^{c}-P^{c}$. Then $\left(  \Lambda d+B^{\ast}\right)  P^{c}f=0
$.

We will compute $B^{\ast}$ on $\Lambda^{0}\left(  M\right)  $: let $\varphi
\in\Lambda^{0}\left(  M\right)  ,\psi\in\Lambda^{1}\left(  M\right)  $. Since
$B\alpha=-J^{-1}AJ\alpha=-J^{-1}\beta\wedge J\alpha$, we have%
\begin{equation}
\left\langle \left\langle B\varphi,\psi\right\rangle \right\rangle =\int
_{M}\varphi J\beta\wedge\ast\psi=\left\langle \left\langle \varphi,B^{\ast
}\psi\right\rangle \right\rangle =\int_{M}\varphi\ast B^{\ast}\psi\label{B*}%
\end{equation}
and it follows that%
\[
B^{\ast}\psi=\ast\left(  J\beta\wedge\ast\psi\right)  ,\ \psi\in\Lambda
^{1}\left(  M\right)  .
\]
In particular $B^{\ast}P^{c}f=\ast\left(  J\beta\wedge\ast P^{c}f\right)
=\ast\left\langle J\beta,P^{c}f\right\rangle $ and from (\ref{Lambda dPcf})
and (\ref{Prod sc}) we obtain%
\begin{equation}
\left(  \Lambda d+B^{\ast}\right)  P^{c}f=0.\label{Lambdad+P etoile}%
\end{equation}

Step3. $\left(  P^{c}\right)  ^{\#}P^{c}f=0$ where $\left(  P^{c}\right)
^{\#}=-\ast P^{c}\ast$.

We have%
\begin{equation}
\left(  P^{c}\right)  ^{\#}=-\ast\left(  d^{c}-B\right)  \ast=\left(
d^{c}\right)  ^{\ast}+B^{\ast}=\left(  d^{c}-B\right)  ^{\ast}+2B^{\ast
}=\left(  P^{c}\right)  ^{\ast}+2B^{\ast}.\label{Pcdiez}%
\end{equation}

Since $M$ is K\"{a}hler, by using (\ref{Pcdiez}) we have%
\[
\left[  d,\Lambda\right]  =-\left(  d^{c}\right)  ^{\ast}=-\left(
P^{c}\right)  ^{\#}+B^{\ast}%
\]
so%
\[
\left(  P^{c}\right)  ^{\#}P^{c}f=\left(  -\left[  d,\Lambda\right]  +B^{\ast
}\right)  P^{c}f=\left(  \Lambda d+B^{\ast}\right)  P^{c}f.
\]

From (\ref{Lambdad+P etoile}) we conclude that
\begin{equation}
\left(  P^{c}\right)  ^{\#}P^{c}f=0.\label{PcdiezPc=0}%
\end{equation}

Step 4. $\left\Vert df\right\Vert =\left\Vert f\beta\right\Vert $.

By (\ref{Pcdiez}) and (\ref{PcdiezPc=0}) we have
\begin{equation}
\left\langle \left\langle \left(  P^{c}\right)  ^{\#}P^{c}f,f\right\rangle
\right\rangle =\left\langle \left\langle \left(  \left(  P^{c}\right)  ^{\ast
}+2B^{\ast}\right)  P^{c}f,f\right\rangle \right\rangle =\left\Vert
P^{c}f\right\Vert ^{2}+2\left\langle \left\langle P^{c}f,Bf\right\rangle
\right\rangle =0.\label{Prod Pc diez Pc}%
\end{equation}
But%
\begin{align}
\left\langle \left\langle P^{c}f,Bf\right\rangle \right\rangle  &
=\left\langle \left\langle P^{c}f,fJ\beta\right\rangle \right\rangle
=\left\langle \left\langle P^{c}f,fJ\beta\right\rangle \right\rangle
=\left\langle \left\langle -JPf,fJ\beta\right\rangle \right\rangle \nonumber\\
& =\left\langle \left\langle -Pf,f\beta\right\rangle \right\rangle
=\left\langle \left\langle -\left(  d+A\right)  f,Af\right\rangle
\right\rangle =-\left\langle \left\langle df,Af\right\rangle \right\rangle
-\left\Vert Af\right\Vert ^{2}\label{Prod Pc B}%
\end{align}
and%
\begin{equation}
\left\langle \left\langle df,Af\right\rangle \right\rangle =\int_{M}%
fdf\wedge\ast\beta=\frac{1}{2}\int_{M}df^{2}\wedge\ast\beta=-\frac{1}{2}%
\int_{M}f^{2}d\left(  \ast\beta\right)  =0\label{Prod d A}%
\end{equation}
because $\beta$ is harmonic and%
\[
\left\Vert d\left(  \ast\beta\right)  \right\Vert =\left\Vert d^{\ast}%
\beta\right\Vert =0.
\]
From (\ref{Prod Pc diez Pc}), (\ref{Prod Pc B}) and (\ref{Prod d A}) it
follows that%
\begin{equation}
\left\Vert P^{c}f\right\Vert ^{2}-2\left\Vert Af\right\Vert ^{2}%
=0\label{Norm Pcf}%
\end{equation}

But
\[
\left\Vert P^{c}f\right\Vert ^{2}=\left\Vert Pf\right\Vert ^{2}=\left\langle
\left\langle \left(  d+A\right)  f,\left(  d+A\right)  f\right\rangle
\right\rangle =\left\Vert df\right\Vert ^{2}+\left\Vert Af\right\Vert ^{2}%
\]
and by replacing this expression of $\left\Vert P^{c}f\right\Vert ^{2}$ in
(\ref{Norm Pcf}) we complete the proof of step 4.

Step 5. $f=0$ and the case $\underset{M}{\sup}\ast\left\langle \beta
,\beta\right\rangle <\lambda_{\Delta}^{1}.$

Since%

\[
\left\Vert df\right\Vert ^{2}=\left\langle \left\langle df,df\right\rangle
\right\rangle =\left\langle \left\langle d^{\ast}df,f\right\rangle
\right\rangle =\left\langle \left\langle \Delta f,f\right\rangle
\right\rangle
\]
and%
\[
\left\Vert f\beta\right\Vert ^{2}=\left\Vert Af\right\Vert ^{2}=\left\langle
\left\langle A^{\ast}Af,f\right\rangle \right\rangle
\]
by the step 4 it follows that%
\[
\left\langle \left\langle \left(  \Delta-A^{\ast}A\right)  f,f\right\rangle
\right\rangle =0
\]
which implies $f=0$.

Finally, as in the computation (\ref{B*}) of $B^{\ast}$we obtain
\[
A^{\ast}\psi=\ast\left\langle \beta,\psi\right\rangle ,\ \psi\in\Lambda
^{1}\left(  M\right)
\]
and so%
\[
A^{\ast}Af=\ast f\left\langle \beta,\beta\right\rangle .
\]
In particular%
\[
\left\langle \left\langle \left(  \Delta-A^{\ast}A\right)  f,f\right\rangle
\right\rangle =\left\langle \left\langle \Delta f,f\right\rangle \right\rangle
-\left\langle \left\langle \ast f\left\langle \beta,\beta\right\rangle
,f\right\rangle \right\rangle \geq\left(  \lambda_{\Delta}^{1}-\underset
{M}{\sup}\ast\left\langle \beta,\beta\right\rangle \right)  \left\Vert
f\right\Vert ^{2}.
\]
So if $\left\Vert \beta\right\Vert _{\infty}^{2}<\lambda_{\Delta}^{1}$ the
operator $\Delta-A^{\ast}A$ is positive definite and the Theorem is proved.
\end{proof}

\ \ \ \ \ \ \ \ \ \ \ \ \ \newline

\subsection{Infinitesimal rigidity results for Levi flat hypersurfaces\bigskip
}

\ \ \ \ \ \ \ \ \ \ \ \newline

By using Corollary \ref{Tangent Inclus Cohomologie} and Corollary
\ref{MC Levi flat cor} it is natural to give the following definition:

\begin{definition}
Let $L$ be a Levi flat hypersurface in a complex manifold $M$. We say that $L
$ is infinitesimally rigid (respectively strongly infinitesimally rigid), if
for any family $\left(  a_{t}\right)  _{t\in I}$ in $\mathcal{V}_{\mathcal{F}%
}$ defining a Levi-flat deformation of $L$
\[
\left[  \frac{d\alpha_{a_{t}}}{dt}_{\left\vert t=0\right.  }\right]  =0\in
H^{1}\left(  \mathcal{Z}\left(  L\right)  ,\delta\right)  ,
\]
respectively%
\[
\frac{d\alpha_{a_{t}}}{dt}_{\left\vert t=0\right.  }=0.
\]

\end{definition}

\begin{theorem}
\label{trans par implies rigid}Let $M$ be a smooth complex manifold and $L$ a
compact connected transversally parallelizable compact Levi flat hypersurface
in $M$. Then $L$ is strongly infinitesimally rigid.
\end{theorem}

\begin{proof}
Since $L$ is transversally parallelizable, every leaf of the Levi foliation is
compact or every leaf of the Levi foliation is dense (see \cite{Godbillon91}
for example for the properties of transversally parallelizable manifolds). By
Proposition \ref{transversally par} we can consider a DGLA defining couple
$\left(  \gamma,X\right)  $ such that $\mathfrak{b=}\iota_{X}d\gamma=0$ and
$\delta=d_{b}$.

Let $\left(  a_{t}\right)  _{t\in I}$ be a family in $\mathcal{V}%
_{\mathcal{F}} $ defining a Levi-flat deformation of $L$ and $p=\frac{da_{t}%
}{dt}_{\left\vert t=0\right.  }$. Then (\ref{Eq p A}) becomes
\begin{equation}
d_{b}d_{b}^{c}p=0.\label{Eq p B}%
\end{equation}

Suppose that every leaf of the Levi foliation of $L$ is compact. By
(\ref{Eq p B}) it follows that $p$ is constant on each leaf, so $\delta
^{c}p\ =0$. By Proposition \ref{d/dt(alfa(a))=delta c p} it follows that $L$
is strongly infinitesimally rigid.

Suppose now that every leaf of the Levi foliation is dense. Let $z_{0}\in L$
such that $p\left(  z_{0}\right)  =\underset{L}{\sup}p$ and let $L_{z_{0}}$
the leaf of the Levi foliation through $z_{0}$. By (\ref{Eq p B}) it follows
that $p$ is constant on $L_{z_{0}}$. Since $L_{z_{0}}$ is dense, $p$ is
constant on $L$ and $L$ is strongly infinitesimally rigid.
\end{proof}

Now we study the case of infinitesimal rigidity of general Levi flat
hypersurfaces in smooth compact connected K\"{a}hler manifolds.

\begin{lemma}
\label{Espace de p}Let $M$ be a $n$-dimensional K\"{a}hler manifold, $L$ a
Levi flat hypersurface in $M$ and $F$ be a compact leaf of the Levi foliation.
Let $\left(  a_{t}\right)  _{t\in I}$ a family in $\mathcal{V}_{\mathcal{F}}$
defining a Levi-flat deformation of $L$ and $p=\frac{da_{t}}{dt}_{\left\vert
t=0\right.  }$. Then%
\[
\int_{F}p\mathfrak{b}_{F}\wedge J\mathfrak{b}_{F}\wedge\omega^{n-2}=0\
\]
where $\omega$ is a K\"{a}hler form on $M$ and $J$ the complex structure of
$M$.
\end{lemma}

\begin{proof}
From (\ref{Eq p K}) it follows%
\[
\int_{F}\partial_{b}\overline{\partial}_{b}p\wedge\omega^{n-2}+\int
_{F}\partial_{b}p\wedge\overline{\theta}_{F}\wedge\omega^{n-2}-\int
_{F}\overline{\partial}_{b}p\wedge\theta_{F}\wedge\omega^{n-2}-\int
_{F}p\overline{\theta}_{F}\wedge\theta_{F}\wedge\omega^{n-2}=0.
\]

Since $\partial_{b}\theta_{F}=\overline{\partial}_{b}\theta_{F}=0$, we have%
\[
\int_{F}\partial_{b}\overline{\partial}_{b}p\wedge\omega^{n-2}=\int_{F}%
d_{b}\left(  \overline{\partial}_{b}p\wedge\omega^{n-2}\right)  =0,
\]%
\[
\int_{F}\partial_{b}p\wedge\overline{\theta}_{F}\wedge\omega^{n-2}=\int
_{F}\partial_{b}\left(  p\overline{\theta}_{F}\right)  \wedge\omega^{n-2}%
=\int_{F}d_{b}\left(  p\overline{\theta}_{F}\wedge\omega^{n-2}\right)  =0,
\]%
\[
\int_{F}\overline{\partial}_{b}p\wedge\theta_{F}\wedge\omega^{n-2}=\int
_{F}\overline{\partial}_{b}\left(  p\theta_{F}\right)  \wedge\omega^{n-2}%
=\int_{F}d_{b}\left(  p\theta_{F}\wedge\omega^{n-2}\right)  =0
\]
and the lemma is proved.
\end{proof}

\begin{theorem}
\label{rigidity}Let $M$ be a $n$-dimensional K\"{a}hler manifold, $J$ the
complex structure of $M$, $\omega$ a K\"{a}hler form on $M$ and $L$ a Levi
flat hypersurface in $M$ with compact leaves. Suppose that for every leaf $F$
of the Levi foliation such that $\mathfrak{b}_{F}\neq0$, $\Delta_{F}-T_{F}$ is
positive definite on $\mathfrak{B}_{F}$, where $\Delta_{F}$ is the Laplace
operator on $F$, $T_{F}\in End\left(  \Lambda^{0}\left(  F\right)  \right)  $
\ is the operator defined by $T_{F}\varphi=\ast\varphi\left\langle
\mathfrak{b}_{F},\mathfrak{b}_{F}\right\rangle $ and
\[
\mathfrak{B}_{F}\mathfrak{=}\left\{  f\in C^{\infty}\left(  M\right)
:\ \int_{F}f\mathfrak{b}_{F}\wedge J\mathfrak{b}_{F}\wedge\omega
^{n-2}=0\right\}  .
\]
Then $L$ is strongly infinitesimally rigid. In particular this is true if
$\left\Vert \mathfrak{b}_{F}\right\Vert _{\infty}^{2}<\lambda_{F}$ for every
leaf $F$ of $L$, where $\lambda_{F}$ is the smallest strictly positive
eigenvalue of the Dirichlet form $f\mapsto\int_{F}$ $\left\vert
\bigtriangledown f\right\vert ^{2}$ restricted to $\mathfrak{B}_{F}$ and
$\left\Vert \mathfrak{b}_{F}\right\Vert _{\infty}^{2}=\underset{F}{\sup}%
\ast\left\langle \mathfrak{b}_{F},\mathfrak{b}_{F}\right\rangle $.
\end{theorem}

\begin{proof}
Let $\left(  a_{t}\right)  _{t\in I}$ a family in $\mathcal{V}_{\mathcal{F}}$
defining a Levi-flat deformation of $L$ and $p=\frac{da_{t}}{dt}_{\left\vert
t=0\right.  }$. Let $F$ be a leaf of the Levi foliation. We recall that by
(\ref{rel delta db}) we have $\delta\alpha=d_{b}\alpha+\mathfrak{b}_{F}%
\wedge\alpha$.

If $\mathfrak{b}_{F}=0$, (\ref{Eq p 1}) implies that $dd^{c}p=0$ and it
follows that $p$ is constant on $F.$

Suppose now that $\mathfrak{b}_{F}\neq0$. By (\ref{Eq p 1}) we have
$\delta\delta^{c}p=0$ and by Lemma \ref{Espace de p} $p\in\mathfrak{B}_{F}$.
We can apply the uniqueness Theorem \ref{uniqueness} on $F$ for $\beta
=\mathfrak{b}_{F}$ and it follows that $p=0$ on $F$.

So $\delta^{c}p=0$ on $L$ \ and by Proposition \ref{d/dt(alfa(a))=delta c p}%
\ $L$ is strongly infinitesimally rigid. The last assertion follows also by
Theorem \ref{uniqueness}.
\end{proof}

\begin{remark}
Note that in general $\mathfrak{b}_{F}$ is not continuous with respect to $F$.
\end{remark}

\ \ \ \ \ \ \ \ \ \ \ \newline

\subsection{Non existence of Levi flat transversally parallelizable
hypersurfaces in $\mathbb{CP}_{n}$, $n\geq2$}

$\ \ \ \ \ \ \ $

One of the basic questions in the theory of foliations is the following: Let
$\mathcal{F}$ be a singular holomorphic foliation of codimension $1$ of
$\mathbb{CP}_{2}$. Does every leaf of $\mathcal{F}$ accumulate to the singular
set of $\mathcal{F}$? This question led to the conjecture of the non-existence
of smooth Levi flat hypersurfaces in $\mathbb{CP}_{n}$, $n\geqslant2$, and
under suitable hypothesis, in compact complex manifolds.

We recall that for $\mathbb{CP}_{n}$, $n\geqslant3$, the positive answer to
this question was given in \cite{Neto99} and \cite{Siu00}. For $n=2$ the
problem is still open. In this paragraph we prove the non existence of
transversally paralelisable Levi flat hypersurfaces in:

a) connected complex manifolds $M$ such that for every $p\neq q\in M$ and
every real hyperplane $H_{q}$ in $T_{q}M$ there exists a holomorphic vector
field $Y$ on $M$ such that $Y\left(  p\right)  =0$ and $Y\left(  q\right)
\oplus H_{q}=T_{q}M$ (Theorem \ref{Non exist trans par}). The proof uses
techniques developped in this paper.

b) complex compact K\"{a}hler surfaces $M$ such that $\dim H^{2}\left(
M\right)  =1$ (Theorem \ref{Levi plat in H2=0}). The proof of this result was
communicated to us by M. Brunella \cite{Brunella10}.

Both theorems \ref{Non exist trans par} and \ref{Levi plat in H2=0} imply that
there are no transversally paralelisable Levi flat hypersurfaces in
$\mathbb{CP}_{2}$ (Theorem \ref{Non exist transpar CP2}).

\begin{theorem}
\label{Non exist transpar CP2}There are no transversally parallelizable
$C^{2}$ Levi flat hypersurfaces in $\mathbb{CP}_{n}$, $n\geq2$.
\end{theorem}

\begin{proof}
Recall that Y.-T. Siu's theorem \cite{Siu02} \ and \cite{Iordan08} prove the
non existence of $C^{2}$ Levi flat hypersurfaces in $\mathbb{CP}_{n}$,
$n\geq3$.

Let $L$ be a transversally parallelizable Levi flat hypersurface in
$\mathbb{CP}_{2}$. Suppose that $Y$ is a holomorphic vector field on $M$.
Then$\left(  \Phi_{t}^{Y}\left(  L\right)  \right)  _{t}$ is a Levi-flat
deformation of $L$ and let $\left(  a_{t}\right)  _{t\in I}$ a family in
$\mathcal{V}_{\mathcal{F}}$ defining this Levi-flat deformation of $L$. Set
$p=\frac{da_{t}}{dt}_{\left\vert t=0\right.  }$.

By (\ref{d/dt afa (t) 4}) we have%
\[
\frac{d}{dt}_{\left\vert t=0\right.  }\alpha_{a_{t}}=\delta^{c}p\ .
\]

Theorem \ref{trans par implies rigid} \ implies that $L$ is strongly
infinitesimally rigid and it follows that $\delta^{c}p=0$. By Lemma
\ref{transversally par}, we may suppose that $\delta=d_{b}$, so $d_{b}^{c}p=0$.

As a Levi flat hypersurface in $\mathbb{CP}_{2}$ has no compact leaves, every
leaf is dense in $L$ and it follows that $p$ is constant.

Let $g$ be a fixed Hermitian metric on $\mathbb{CP}_{2}$ and $Z=grad_{g}%
r/\left\Vert grad_{g}r\right\Vert _{g}^{2}$. As in \ref{paragraph},
$a_{t}\left(  X\right)  =r\left(  X\left(  t\right)  \right)  $,
$X\in\mathbb{CP}_{2}$ with $X\left(  t\right)  =\gamma_{Z,X}\cap\Phi_{t}%
^{Y}\left(  L\right)  $ and $\gamma_{Z,X}$ the integral curve of $Z$ passing
through $X$. We have%
\[
Y=Y_{n}+Y_{t}%
\]
where%
\[
Y_{n}=dr\left(  Y\right)  Z,\ Y_{t}\left(  r\right)  =Y-dr\left(  Y\right)  Z
\]
are the normal and tangential components of $Y$. \ Since $a_{t}\left(
X\right)  =r\left(  \Phi_{t}^{Y_{n}}\left(  X\right)  \right)  $ it follows
that%
\[
p=\frac{da_{t}}{dt}_{\left\vert t=0\right.  }=dr\left(  Y_{n}\right)
=Y_{n}\left(  r\right)  .
\]
As $Y_{n}=\left\langle Z,Y\right\rangle _{g}Z$, where $\left\langle
\cdot,\cdot\right\rangle _{g}$ is the scalar product induced by $g$ we obtain
that $p=\left\langle Z,Y\right\rangle _{g}$ and we conclude that $\left\langle
Z,Y\right\rangle _{g}$ is constant on $L$ for every holomorphic vector field
on $M$.

Let $X\in L$ and consider homogeneous coordinates $\left[  z_{0},z_{1}%
,z_{2}\right]  $ in $\mathbb{CP}_{2}$ such that $X=\left[  1,0,0\right]  $ and
the Euler vector field $Y$ such that $Y\left(  \left[  1,0,0\right]  \right)
=0$. \ Since
\[
\left\langle Z,Y\right\rangle _{g}\left(  X\right)  =\left\langle Z\left(
\left[  1,0,0\right]  \right)  ,Y\left[  1,0,0\right]  \right\rangle _{g}=0
\]
it follows that $\left\langle Z,Y\right\rangle _{g}=0$ and this means that $Y
$ is tangent to $L$. \ But by Siu's Proposition 2.3 \cite{Siu02} this gives a contradiction.
\end{proof}

This theorem can be generalized and proved without using Y.-T. Siu's
Proposition 2.3 from \cite{Siu02} :

\begin{theorem}
\label{Non exist trans par}Let $M$ be a connected complex manifold such that
for every $p\neq q\in M$ and every real hyperplane $H_{q}$ in $T_{q}M$ there
exists a holomorphic vector field $Y$ on $M$ such that $Y\left(  p\right)  =0$
and $Y\left(  q\right)  \oplus H_{q}=T_{q}M$. \ Than there are no compact
transversally parallelizable Levi flat hypersurfaces in $M$. The hypothesis
are fulfilled if $M=\mathbb{CP}_{n}$, $n\geq2$.
\end{theorem}

\begin{proof}
Let $L$ be a transversally parallelizable Levi flat hypersurface in $M$. As in
the proof of theorem \ref{Non exist transpar CP2}, $\left\langle
Z,Y\right\rangle _{g}$ is constant on every leaf of the Levi foliation of $L$
for every holomorphic vector field on $M$. Let $p\in L$ and $q$ a distinct
point of the leaf $F$ passing through $p$. Let $Y$ be a holomorphic vector
field on $M$ such that $Y\left(  p\right)  =0$ and $Y\left(  q\right)  \oplus
T_{q}L=T_{q}M$. \ Since $Y\left(  p\right)  =0$ it follows that $Y$ is tangent
to $F$ and we obtain a contradiction.
\end{proof}

\begin{lemma}
\label{dimH2=1}Let $L$ be a real hypersurface in a complex compact K\"{a}hler
surface $M$ such that $M\backslash L=U_{1}\cup U_{2}$ where $U_{1},U_{2} $ are
open disjoint subsets of $M$ and let $\omega$ be the $\left(  1,1\right)
$-form associated to the K\"{a}hler metric of $M$. Suppose that $\dim
H^{2}\left(  M\right)  =1$. Then

i) $\omega$ is exact on $U_{1}$ or on $U_{2}$;

ii) the restriction of $\omega$ to $L$ is exact.
\end{lemma}

\begin{proof}
i) Let $\psi$ be a cycle such that $H^{2}\left(  M\right)  =\mathbb{C}\left[
\psi\right]  $. Suppose that $\omega$ is neither exact on $U_{1}$ nor on
$U_{2}$. Then there exist 2-cycles $\varphi_{j}\subset U_{j}$ such that
$\int_{\varphi_{j}}\omega\neq0$, $j=1,2$. But $\left[  \varphi_{j}\right]
=c_{j}\left[  \psi\right]  $, $j=1,2$ and $\left[  \varphi_{1}\right]  \left[
\varphi_{2}\right]  =0$. Contradiction.

ii) Suppose that $\omega$ is exact on $U_{1}$. Let $\varphi$ be a $2$-cycle
$\varphi$ on $L$. We can aproximate $\varphi$ by $2$-cycles $\varphi
_{\varepsilon}$ on $U_{1}$. Since $\int_{\varphi_{\varepsilon}}\omega=0$, it
folllows that $\int_{\varphi}\omega=0$,.
\end{proof}

\begin{corollary}
\label{Int=0}Under the hypothesis of Lemma \ref{dimH2=1} \ we have $\int
_{L}\gamma\wedge\omega=0$ for every closed $1$-form $\gamma$ on $L$.
\end{corollary}

\begin{proof}
By Lemma \ref{dimH2=1}, $\omega=d\alpha$ on $L$, so%
\[
\int_{L}\gamma\wedge\omega=\int_{L}d\left(  \gamma\wedge\alpha\right)  =0.
\]

\end{proof}

\begin{theorem}
\label{Levi plat in H2=0}Let $L$ be a real hypersurface in a complex compact
K\"{a}hler surface $M$ such that $M\backslash L=U_{1}\cup U_{2}$ where
$U_{1},U_{2}$ are open disjoint subsets of $M$ such that $\dim H^{2}\left(
M\right)  =1$. There are no transversally parallelizable Levi flat
hypersurfaces in $M$.
\end{theorem}

\begin{proof}
Let $\omega$ be the $\left(  1,1\right)  $-form associated to the K\"{a}hler
metric of $M$. Let $L$ be a Levi flat transversally parallelizable
hypersurface in $M$ such that the Levi foliation of $L$ is given by the 1-form
$\gamma$. Then $\gamma\wedge\omega\left(  x\right)  \neq0$ for every $x\in L$.
Indeed, let $x\in L$ and choose local coordinates $\left(  t_{1},t_{2}%
,t_{3}\right)  $ in a neighborhood of $x$ such that $x=0$, $\gamma
=\alpha\left(  t_{1}\right)  dt_{1}$ and $\left(  0,t_{2},t_{3}\right)  $ are
coordinates on the leaf $L_{x}$ through $x$. There exists local holomorphic
coordinates $z\in\left(  z_{1},z_{2}\right)  $ in a neighborhood $V$ of $x$
such that $L_{x}=\left\{  z\in V:\ z_{2}=0\right\}  $. It follows that
$\alpha\left(  0\right)  dt_{1}\wedge dz_{1}\wedge d\overline{z}_{1}\neq
0$.\ Consequently $\int_{L}\gamma\wedge\omega>0$ or $\int_{L}\gamma
\wedge\omega<0$ and we obtain a contradiction by Corollary \ref{Int=0}.
\end{proof}

\textbf{Acknowledgement. }\textit{We would like to thank the referee of our
paper who indicated us the reference \cite{Kodaira61}, noticed an error in the
first version of the manuscript and did a lot of remarks that improved the
quality of the paper. He also remarked that the proof of Theorem
\ref{Non exist transpar CP2} follows by the fact that a closed form defining
the Levi foliation of }$L$\textit{\ defines a holonomy invariant
Lebesgue-class measure on transersals. Thus, by a Theorem of D. Sullivan
\cite{Sullivan76}, there is a closed current }$T\neq0$\textit{\ directed by
the Levi foliation}$.$\textit{\ But as }$L$\textit{\ can be isotoped on
}$L^{\prime}$,\textit{\ which is still foliated and disjoint of }%
$L,$\textit{\ we obtain a contradiction. The authors would also like to thank
T.-C. Dinh for useful discussions.}

\renewcommand\baselinestretch{1}
\bibliographystyle{amsplain}
\bibliography{deformations}

\end{document}